\documentclass{article}

\newcommand{\bpymemail}{pym@maths.ox.ac.uk}

\usepackage{mathscinet}
\usepackage{array}
\usepackage{caption}
\usepackage{amsthm,amsfonts,amssymb,amsmath} 
\usepackage{aliascnt} 
\usepackage{mathrsfs} 
\usepackage{xargs,ifthen} 
\usepackage{enumitem}

\usepackage{color}
\usepackage{xypic} 

\usepackage{hyperref}
\definecolor{sxdarkblue}{RGB}{150,180,200} 
\definecolor{sxlightblue}{RGB}{186,215,230} 
\definecolor{sxred}{RGB}{153,0,0} 
\definecolor{tocolor}{rgb}{.1,.1,.1}
\definecolor{urlcolor}{rgb}{.2,.2,.6}
\definecolor{linkcolor}{rgb}{.1,.1,.5}
\definecolor{citecolor}{rgb}{.4,.2,.1}
\hypersetup{backref=true, colorlinks=true, urlcolor=urlcolor, linkcolor=linkcolor, citecolor=citecolor}

\newcommandx{\thdef}[2]{
	\newaliascnt{#1}{theorem}  
	\newtheorem{#1}[#1]{#2}
	\aliascntresetthe{#1}  
	\newtheorem*{#1*}{#2}
	\expandafter\newcommand\expandafter{\csname #1autorefname\endcsname}{#2}
}

\newtheorem{theorem}{Theorem}[section]

\thdef{lemma}{Lemma}
\thdef{corollary}{Corollary}
\thdef{problem}{Problem}
\thdef{proposition}{Proposition}

\theoremstyle{definition}
\thdef{definition}{Definition}

\theoremstyle{remark}
\thdef{remark}{Remark}
\thdef{example}{Example}

\newcommand{\defn}[1]{\textbf{\textit{#1}}} 
\newcommand{\spc}[1]{\mathsf{#1}} 
\newcommand{\shf}[1]{\mathcal{#1}} 

\newcommand{\CC}{\mathbb{C}}
\newcommand{\ZZ}{\mathbb{Z}}

\newcommand{\PP}[1][]{\mathbb{P}^{#1}}

\newcommand{\rbrac}[1]{\left(#1\right)} 
\newcommand{\abrac}[1]{\left\langle#1\right\rangle} 

\newcommand{\mapdef}[5]{
	\begin{array}{ccccc}
	#1 &:& #2 &\to& #3 \\
		&&  #4 &\mapsto& #5
	\end{array}
}
\newcommandx{\fn}[2][2=]{#1\ifthenelse{\equal{#2}{}}{}{\!\rbrac{{#2}}}} 

\newcommand{\ext}[2][\bullet]{\spc{\Lambda}^{#1}{#2}} 
\newcommandx{\sym}[3][1=\bullet,3=]{\spc{Sym}^{#1}_{#3}{#2}} 
\newcommandx{\Aut}[2][1=]{\fn{\spc{Aut}_{#1}}[#2]} 
\newcommandx{\image}[1]{\fn{\spc{img}}[#1]} 

\newcommandx{\hlgy}[3][1=\bullet,3=]{\spc{H}_{#1}^{#3}\!\rbrac{{#2}}} 
\newcommandx{\cohlgy}[3][1=\bullet,3=]{\spc{H}^{#1}_{#3}\!\rbrac{{#2}}} 
\newcommandx{\Pic}[2][1=]{\fn{\spc{Pic}_{#1}}[{#2}]} 

\newcommandx{\sHom}[2][1=]{ \fn{ \shf{H}om_{#1}}[{#2}] } 
\newcommand{\sO}[1]{\shf{O}_{#1}}
\newcommand{\Rforms}[2][\bullet]{\forms[#1,reg]{#2}}
\newcommandx{\forms}[2][1=\bullet]{\Omega^{#1}_{#2}} 
\newcommandx{\can}[1][1=]{\mathcal{K}_{#1}} 
\newcommandx{\acan}[1][1=]{\mathcal{K}_{#1}^{\vee}} 
\newcommandx{\tshf}[1]{\shf{T}_{#1}} 
\newcommandx{\der}[2][1=\bullet]{\mathscr{X}^{#1}_{#2}} 

\newcommandx{\lie}[2][2=]{\fn{\mathscr{L}_{#1}}[#2]} 
\newcommandx{\hook}[2][2=]{\fn{i_{#1}}[#2]} 
\newcommand{\cvf}[1]{\partial_{{#1}}} 
\newcommand{\deriv}[2]{\frac{d#1}{d#2}} 

\newcommandx{\Dgn}[2][1=]{\fn{\spc{Dgn}_{#1}}[#2]}
\newcommandx{\Sec}[2][1=]{\fn{\spc{Sec}_{#1}}[#2]}

\newcommand{\bL}{\spc{L}}
\newcommand{\bS}{\spc{S}}
\newcommand{\W}{\spc{W}}
\newcommand{\X}{\spc{X}}
\newcommand{\Y}{\spc{Y}}
\newcommand{\U}{\spc{U}}
\newcommand{\V}{\spc{V}}
\newcommand{\C}{\spc{C}}
\newcommand{\Q}{\spc{Q}}
\newcommand{\tD}{\widetilde{\spc{D}}}
\newcommand{\D}{\spc{D}}
\newcommand{\T}{\spc{T}}
\newcommand{\E}{\spc{E}}
\newcommand{\SU}[1]{\fn{\spc{SU}}[#1]}
\renewcommand{\H}{\spc{H}}

\newcommand{\sE}{\shf{E}}
\newcommand{\sI}{\shf{I}} 
\newcommand{\sJ}{\shf{J}}
\newcommand{\sF}{\shf{F}}
\newcommand{\sN}{\shf{N}}

\newcommand{\sA}{\shf{A}}

\newcommand{\tE}{\widetilde{E}}
\newcommand{\ps}{\sigma} 
\newcommand{\pss}{\ps^\sharp} 
\newcommand{\Res}{\mathrm{Res}\,} 
\newcommand{\Pois}[1]{\fn{\spc{Pois}}[{#1}]} 
\newcommand{\Pf}{\mathrm{Pf}} 
\newcommand{\hess}[1]{Hess(#1)}

\usepackage{thmtools}

\usepackage{titling}

\usepackage{abstract}

\let\oldbibliography\thebibliography
\renewcommand{\thebibliography}[1]{\oldbibliography{#1}
\setlength{\itemsep}{0pt}} 

\newcommand{\papertitle}{Elliptic singularities on log symplectic manifolds and Feigin--Odesskii Poisson brackets}
\hypersetup{pdftitle={\papertitle},pdfauthor={Brent Pym}}

\begin{document}

\title{\papertitle}

\author{Brent Pym\thanks{Mathematical Institute and Jesus College, University of Oxford, \href{mailto:\bpymemail}{\texttt{\bpymemail}}}}

\maketitle

\begin{abstract}
A log symplectic manifold is a complex manifold equipped with a complex symplectic form that has simple poles on a hypersurface.  The possible singularities of such a hypersurface are heavily constrained.  We introduce the notion of an elliptic point of a log symplectic structure, which is a singular point at which a natural transversality condition involving the modular vector field is satisfied, and we prove a local normal form for such points that involves  the simple elliptic surface singularities $\tE_6,\tE_7$ and $\tE_8$.  Our main application is to the classification of Poisson brackets on Fano fourfolds.  For example, we show that Feigin and Odesskii's Poisson structures of type $q_{5,1}$ are the only log symplectic structures on projective four-space whose singular points are all elliptic.
\end{abstract}

\tableofcontents

\newpage
\section{Introduction}

During the past twenty years, the classification of projective   Poisson varieties of low dimension has seen considerable progress; we note, in particular, the works on  surfaces~\cite{Bartocci2005,Ingalls1998}, threefolds with finitely many zero-dimensional symplectic leaves~\cite{Druel1999}, and Fano threefolds with Picard rank one~\cite{Cerveau1996,Loray2013,Polishchuk1997}. 

But in higher dimensions---even for simple manifolds such as the projective space $\PP[4]$---rather less is known.  One of the main difficulties is the increase in the complexity of the symplectic foliation: in dimension four, one encounters, for the first time, the possibility that the foliation could have leaves of three different dimensions (zero, two or four).  It is therefore useful to impose constraints that simplify the local structure of the foliation, which we do in this paper by focusing on the class of ``log symplectic'' manifolds.

A log symplectic manifold is simply a real or complex manifold $\X$, equipped with a symplectic form $\omega$ that has simple poles along a hypersurface $\D\subset \X$.  Put differently, we require that the Poisson bracket on $\X\setminus \D$ defined by $\omega$ extends to a Poisson bracket on all of $\X$ whose Pfaffian is a reduced defining equation for $\D$.  Thus, log symplectic manifolds are the Poisson manifolds that are, in some sense, as close as possible to being symplectic.  As a result, they often arise when one attempts to compactify symplectic manifolds; in particular, many interesting moduli spaces in algebraic geometry and gauge theory come equipped with natural log symplectic structures.  Examples include compactified moduli spaces of $\SU{2}$ monopoles~\cite{Atiyah1988,Goto2002}; Hilbert schemes of various commutative surfaces~\cite{Bottacin1998,Ran2013} and their  noncommutative analogues~\cite{Nevins2007}; and certain moduli spaces of vector bundle maps over elliptic curves~\cite{Feigin1998,Polishchuk1998}.

Recently, there has been considerable interest in the general properties of log symplectic manifolds---particularly in the real  $C^\infty$ category, where they are also known as $b$-symplectic, or topologically stable.  Results include descriptions of local normal forms and cohomological invariants~\cite{Dongho2012,Guillemin2014}; obstructions to global existence~\cite{Cavalcanti2013,Marcut2014a}; unobstructedness of deformations~\cite{Marcut2014,Ran2013}; constructions of symplectic groupoids~\cite{Gualtieri2014a}, deformation quantizations~\cite{Nest1996} and Rozansky--Witten invariants~\cite{Goto2002}; classifications up to diffeomorphism~\cite{Radko2002} and Morita equivalence~\cite{Bursztyn2003} in the two-dimensional case; and Delzant-type classifications of the corresponding toric integrable systems~\cite{Gualtieri2014,Guillemin2015}.  Generalized complex analogues of log symplectic structures have also been studied~\cite{Cavalcanti2015,Goto2015}.

Since most of these works take place in the $C^\infty$ category, they assume that the hypersurface $\D$ is smooth, or, occasionally, that it is a union of smooth components with normal crossings.  But in this paper, we are interested in the holomorphic case, where it seems to be rather difficult to find compelling examples for which these assumptions on the hypersurface are satisfied.  In particular, for the  moduli spaces mentioned above, the singularities of the hypersurface $\D$ are rather more complicated. Nevertheless, the unobstructedness of deformations for Hilbert schemes has been established using a resolution of singularities~\cite{Ran2013}.

It is also difficult to find examples on simple varieties of interest, such as projective space.
Since $\D$ is always an anticanonical divisor, it is reasonable to look for log symplectic structures on complex manifolds for which the anticanonical bundle has many holomorphic sections, such as Fano manifolds.  In so doing, we immediately find hypersurfaces that are highly singular:
\begin{theorem}[\cite{Gualtieri2013a}]\label{thm:fano-singular}
Let $(\X,\D,\omega)$ be a holomorphic log symplectic manifold, and let $\D_{sing} \subset \D$ be the singular locus of the degeneracy hypersurface.  Suppose that $\X$ is Fano, or that the polynomial $c_1c_2-c_3 \in \cohlgy[6]{\X,\ZZ}$ in the Chern classes of $\X$ is nonzero.  Then $\D_{sing}$ is nonempty, and all of its irreducible components have codimension at most three in $\X$.  When the codimension of $\D_{sing}$ is exactly three and $\X$ is compact, the formula
\[
[\D_{sing}] =  (c_1c_3-c_3) \cap [\X] \in \hlgy[\dim \X - 6]{\X,\ZZ}
\]
gives the fundamental class of the singular locus.
\end{theorem}

Meanwhile, toric varieties give many examples for which $\D$ is a simple normal crossings divisor. By an inductive argument that reduces the problem to dimension three by intersecting components of $\D$, Lima and Pereira showed that every normal crossings example that is Fano with cyclic Picard group is toric:
\begin{theorem}[\cite{Lima2014}]\label{thm:lima-pereira}
Let $(\X,\D,\omega)$ be a log symplectic Fano manifold of dimension $2n$ where $n >2$, and suppose that $\Pic{\X}\cong \ZZ$.  If the degeneracy hypersurface of $\omega$ is a union of smooth components with normal crossings, then $\X \cong \PP[2n]$ is a projective space, $\D$ is the union of the coordinate hyperplanes, and $\omega$ restricts to an invariant symplectic form on the torus $(\CC^\times)^{2n} = \PP[2n] \setminus \D$
\end{theorem}

In light of these considerations, we are lead to ask what other sorts of singular hypersurfaces $\D \subset \X$ can arise as the degeneracy hypersurface of a log symplectic structure, and to develop techniques for dealing with the singularities when they arise.  In this paper, we initiate the systematic study of this class of hypersurface singularities, focusing on the case in which $\D$ is normal, i.e.~it has no singularities of codimension two in $\X$. The central message of the paper is that elliptic curves play a fundamental role in the geometry---both locally and globally---and that despite the increased complexity of the singularities, rigidity results analogous to \autoref{thm:lima-pereira} still hold true.

After a review of some basic defininitions and examples in \autoref{sec:logforms}, we lay the foundations for the local structure theory of log symplectic manifolds with singular hypersurfaces in \autoref{sec:local}, using K.~Saito's theory of logarithmic forms.  In \autoref{thm:codim3} and \autoref{prop:codim3-converse}, we show that a normal hypersurface  singularity germ $\D \subset \X$ supports a nondegenerate logarithmic two-form if and only if its singular locus is Gorenstein of pure codimension three in $\X$---a skew-symmetric analogue of the Jacobian condition~\cite{Aleksandrov1988,Terao1980} for the freeness of a divisor.  Meanwhile \autoref{thm:germ-moduli}, based on a Moser lemma for the logarithmic de Rham complex, gives an efficient mechanism for producing local normal forms.

In \autoref{sec:ell-ex}, we introduce the notion of an elliptic point of a log symplectic structure: a point where $\D$ has a normal singularity and a natural transversality condition involving the Weinstein modular vector field is satisfied; see \autoref{def:ell-point}.  Ellipticity is an open condition on the one-jet of the Poisson tensor, giving a class of Poisson structures that is stable under deformation. For this reason, we believe that elliptic points are, in some sense, the simplest normal singularities one could hope for.

We justify the use of the word ``elliptic'' by proving a local normal form (\autoref{thm:local-elliptic}), which shows that near an elliptic point, the hypersurface is isomorphic to a product of a smooth space with a simple elliptic surface singularity of type $\tE_6$, $\tE_7$ or $\tE_8$.  The main examples are the Poisson structures $q_{2n+1,1}$ on $\PP[2n]$ introduced by Feigin and Odesskii~\cite{Feigin1989}, for which the hypersurface $\D$ is the union of the secant $(n-1)$-planes of an elliptic normal curve.  In these examples, the generic singularities of $\D$ are of type $\tE_6$.

The paper culminates in \autoref{sec:ell-fano}, where we apply our techniques to the classification of Poisson structures on Fano fourfolds:

\begin{restatable}{theorem}{fanos}\label{thm:p4-ell}
Every purely elliptic log symplectic structure on $\PP[4]$ is isomorphic to a member of Feigin and Odesskii's family $q_{5,1}$.   Moreover, the following Fano fourfolds do not support any purely elliptic log symplectic structures:
\begin{itemize}[noitemsep,topsep=4pt]
\item Smooth quadric or cubic fourfolds $\X \subset \PP[5]$
\item Products of the form $\X = \PP[1]\times \W$, where $\W \subset \PP[4]$ is a Fano hypersurface 
\end{itemize}
\end{restatable}

\begin{remark}
One can show that a smooth hypersurface $\X \subset \PP[n]$ of degree $d$ with $d,n \ge 4$ does not admit any Poisson structures.  This is why we focus on quadric and cubic fourfolds, which do possess nontrivial Poisson structures. \qed
\end{remark}

The method of proof, which can easily be applied to other compact fourfolds, is as follows.  In four dimensions, if all singular points are elliptic, then the singular locus is necessarily a disjoint union of elliptic curves.  Using \autoref{thm:fano-singular} and the local normal form, we can compute the possible degrees of these elliptic curves, and immediately rule out most of the manifolds in question.  But for $\PP[4]$ and the cubic fourfolds, more refined information is required; the hard part of the proof is a detailed analysis of the algebraic geometry in those cases.  

The reader familiar with Feigin and Odesskii's Poisson structures may be wondering why some other examples do not appear.  Indeed, Feigin and Odesskii defined a second family of Poisson structures on $\PP[4]$, called $q_{5,2}$.  Those Poisson structures are log symplectic, and are also associated with elliptic curves, being related to $q_{5,1}$ by a birational automorphism of $\PP[4]$.  However, their degeneracy hypersurfaces are not normal, so they are not elliptic in our sense.

Meanwhile, for the Poisson structures $q_{6,1}$ on $\PP[5]$, the closures of the four-dimensional symplectic leaves give a pencil of cubic fourfolds whose base locus is the secant variety of an elliptic normal sextic.  One might therefore expect these cubics to give elliptic log symplectic manifolds.  However, these cubics are singular, and so there is no contradiction with \autoref{thm:p4-ell}.

\subsubsection*{Acknowledgements}

The author thanks Ragnar-Olaf Buchweitz, Eleonore Faber, Michel Granger and Brian Pike for helpful conversations about logarithmic forms; and Marco Gualtieri, Nigel Hitchin, Colin Ingalls and Vladimir Rubtsov for stimulating discussions about elliptic Poisson structures.  This work was supported by EPSRC Grant EP/K033654/1.

\section{Basic definitions and examples}
\label{sec:logforms}

\subsection{Logarithmic differential forms}

The notion of a differential form with logarithmic singularities along a hypersurface was introduced by Deligne~\cite{Deligne1970} in the normal crossings case, and extended to general hypersurfaces by K.~Saito~\cite{Saito1980}.  In this section, we briefly recall the notions and results from Saito's paper that will be relevant in our study of log symplectic structures.

Here and throughout, $\X$ is a complex manifold and $\D \subset \X$ is a reduced hypersurface; thus, $\D$ may have several irreducible components, but each component is taken with multiplicity one.  We denote by $\D_{sing}$ the singular locus of $\D$, i.e.~the closed analytic subspace defined by the vanishing of the one-jet of a local defining equation for $\D$.  We recall that $\D$ is \defn{normal} if and only if $\dim \D_{sing} \le \dim \D - 2 = \dim \X - 3$.  This condition should not be confused with the \defn{normal crossings} condition, which means that $\D$ is locally isomorphic to the union of a collection of coordinate hyperplanes in $\CC^n$, and therefore has a singular locus of dimension $\dim \X - 2$.

We denote by $\forms{\X}(\D)$ the sheaf of meromorphic  differential forms that are holomorphic on $\X\setminus \D$ and have, at worst, poles of order one on every irreducible component of $\D$. 

\begin{definition}[\cite{Saito1980}]
A meromorphic $k$-form $\omega \in \forms[k]{\X}(\D)$ has \defn{logarithmic singularities along $\D$} if $d\omega \in \forms[k+1]{\X}(\D)$.  The sheaf of $k$-forms with logarithmic singularities is denoted by $\forms[k]{\X}(\log \D)$.
\end{definition}

Just as the holomorphic forms $\forms{\X}$ are dual to the exterior powers of the tangent sheaf $\der{\X} = \ext{\tshf{\X}}$, the dual of $\forms{\X}(\log \D)$ is the sheaf $\der{\X}(-\log\D) \subset \der{\X}$ of \defn{logarithmic multiderivations}---the multiderivations that are tangent to $\D$ in the sense that their action preserves the defining ideal $\sO{\X}(-\D) \subset \sO{\X}$

By definition, the de Rham differential maps logarithmic forms to logarithmic forms.  Hence we have a complex of sheaves $(\forms{\X}(\log \D),d)$, called the \defn{logarithmic de Rham complex}.  The cohomology of this complex is, in general, quite sensitive to the singularities of $\D$.  In \cite{Saito1980}, Saito explains that every logarithmic form $\omega \in \forms[k]{\X}(\log\D)$ has a residue $\Res \omega$, which is a meromorphic form on $\D$ that is holomorphic on the smooth locus.  In this way, we obtain an exact sequence of complexes
\[
\xymatrix{
0 \ar[r] & \forms{\X} \ar[r] & \forms{\X}(\D) \ar[r]^-{\Res} \ar[r] & \Rforms{\D} \ar[r] & 0,
}
\]
where $\Rforms{\D}$ is the image of the residue map.  We recall that $\Rforms{\D}$ may be identified~\cite{Aleksandrov1990} with the ``regular'' differential forms on $\D$ in the sense of \cite{Barlet1978}.  In particular, when $\D$ is normal, we have that $\Rforms{\D} = (\forms{\D})^{\vee\vee}$ is the double dual of the usual differential forms on $\D$.

The sheaf $\forms{\X}(\log \D)$ is a coherent $\sO{\X}$-module, but in general it will not be locally free, i.e.~there will be no holomorphic vector bundle on $\X$ whose sheaf of sections is $\forms{\X}(\log\D)$.  However, $\forms{\X}(\log \D)$ is a reflexive $\sO{\X}$-module, meaning that the canonical map $\forms{\X}(\log \D)\to \forms{\X}(\log \D)^{\vee\vee}$ to the double dual is an isomorphism.  As a result, elements of $\forms{\X}(\log \D)$ exhibit the Hartogs phenomenon~\cite{Hartshorne1980}: sections defined away from an analytic subspace of codimension at least two extend uniquely to all of $\X$.

Amongst all the singular hypersurfaces $\D \subset \X$, the ones for which the sheaf $\forms[1]{\X}(\log \D)$ is locally free are rather special.  They are called \defn{free divisors}.  For example, if $\D$ is smooth or has only normal crossings singularities, then $\D$ is a free divisor.  In this case, we simply have $\forms[k]{\X}(\log \D) \cong \ext[k]{\forms[1]{\X}(\log \D)}$.  In general, the freeness of a divisor is determined by the structure of its singular locus:
\begin{theorem}[\cite{Aleksandrov1988,Terao1980}]\label{thm:jac-free}
A singular reduced hypersurface $\D \subset \X$ is a free divisor if and only if the singular locus $\D_{sing} \subset \D$ is a Cohen--Macaulay space of pure codimension two in $\X$.
\end{theorem}
This theorem implies, in particular, that any plane curve is a free divisor, but isolated hypersurface singularities in $\CC^n$ for $n \ge 3$ are not free.

\subsection{Log symplectic manifolds}
\label{sec:examples}
Let $\X$ be a complex manifold of dimension $2n$ with $n\ge 1$, and let $\D \subset \X$ be a reduced divisor.  We say that a global logarithmic two-form $\omega \in \cohlgy[0]{\X,\forms[2]{\X}(\log \D)}$ is \defn{nondegenerate} if its top power $\omega^n$ is a nonvanishing section of the line bundle $\can[\X](\D)$, where $\can[\X] = \forms[2n]{\X}$ is the canonical bundle of $\X$.  Equivalently, a logarithmic two-form is nondegenerate if and only if the induced map
\begin{align}
\mapdef{\omega^\flat}{\der[1]{\X}(-\log \D)}{\forms[1]{\X}(\log \D)}
{Z}{\hook{Z}\omega} \label{eqn:res}
\end{align}
is an isomorphism of $\sO{\X}$-modules.

The existence of a nonvanishing section of $\forms[2n]{\X}(\D)$ is a logarithmic analogue of the usual Calabi--Yau condition (triviality of the canonical class).  It is equivalent to the divisor $\D$ being an anticanonical divisor, so we make the following
\begin{definition}
A \defn{log Calabi--Yau manifold} is a pair $(\X,\D)$ of a complex manifold and a reduced effective anticanonical divisor $\D \subset \X$.
\end{definition}

A \defn{log symplectic form} on a log Calabi--Yau manifold  $(\X,\D)$ is a nondegenerate global logarithmic two-form $\omega \in \cohlgy[0]{\X,\forms[2]{\X}(\log\D)}$ that is closed, i.e.~$d\omega = 0$.  We remark that the hypersurface $\D$ is completely determined by the meromorphic two-form $\omega$, since it is the polar divisor of $\omega^n$.  For this reason, we will often refer to the pair $(\X,\omega)$ as a log symplectic manifold, the hypersurface $\D$ being implicit.  We will always assume that $\D$ is nonempty, so symplectic manifolds do not count as log symplectic manifolds in our terminology.

\begin{example}[Log symplectic surfaces]
Let $(\X,\D)$ be a log Calabi--Yau surface (e.g.~a Del Pezzo surface containing an elliptic curve).  Since $\forms[3]{\X}(\log \D)=0$, every logarithmic two-form is closed, so that a log symplectic structure on $(\X,\D)$ is determined by an arbitrary nonvanishing section of the line bundle $\can[\X](\D)$. 
 
When $\X = \CC^2$ with coordinates $(x,y)$, every log symplectic structure can be written as
\[
\omega = \frac{dx\wedge dy}{f}
\]
where $f \in \sO{\CC^2}$ is square-free and $\D$ is its zero locus.  Hence $\D$ can be an arbitrary plane curve, with no constraints on the singularities.

When $(\X,\D)$ is compact, the log symplectic form $\omega$ is unique up to rescaling by a constant. The birational classification of projective log symplectic surfaces can be extracted from the full birational classification of projective surfaces containing an effective anticanonical divisor~\cite{Bartocci2005,Ingalls1998}; it turns out that the only possible singularities for $\D$ are of type $A_1$, $A_2$, $A_3$ or $D_4$.\qed 
\end{example}

\begin{example}[Toric examples]
Let $\X = \CC^{2n}$ with coordinates $(x_1,\ldots,x_{2n})$ and let $\D \subset \X$ be the union of the coordinate axes, so that $\X \setminus \D$ is the complex torus $\T=(\CC^\times)^{2n}$.  Let $\lambda = (\lambda_{ij})_{1 \le i,j\le 2n}$ be a skew-symmetric matrix of complex numbers.  Then the two-form
\[
\omega = \sum_{1\le i<j \le 2n} \lambda_{ij} \frac{dx_i}{x_i}\wedge \frac{dx_j}{x_j}
\]
is closed and logarithmic, and is invariant under the action of $\T$.  It will be nondegenerate, defining a log symplectic form on $(\CC^{2n},\D)$, if and only if the matrix $\lambda$ is nonsingular.  If we compactify $\CC^{2n}$ to a projective space by adding a hyperplane at infinity, we obtain the log symplectic structures characterized in \autoref{thm:lima-pereira}.  Similar structures can also be obtained by choosing other toric compactifications of $\T$.\qed
\end{example}

If $(\X,\omega)$ is a log symplectic manifold and $p\in \D$ is a smooth point of the polar divisor, then a version of the Darboux theorem holds in a neighbourhood of $p$.  Namely, there exist coordinates $(x_1,\ldots,x_n,y_1,\ldots,y_n)$ centred at $p$ in which $\D$ is the zero locus of $x_1$ and
\begin{align}
\omega = \frac{dx_1}{x_1}\wedge dy_1 + \sum_{i=2}^n dx_i\wedge dy_i. \label{eqn:log-Darboux}
\end{align}
For a discussion of the proof, see \cite[Lemma 1.2]{Goto2002} or \cite[Proposition 19]{Guillemin2014}.

Every log symplectic form $\omega$ has an inverse, which is a global logarithmic biderivation
\[
\ps \in \cohlgy[0]{\X,\der[2]{\X}(-\log \D)}\subset \cohlgy[0]{\X,\der[2]{\X}}
\]
with Schouten bracket $[\ps,\ps] = 0$.  Thus $\ps$ defines a defines a Poisson bracket on $\sO{\X}$ by the usual formula
\[
\{f,g\} = \abrac{df\wedge dg,\ps}
\]
for $f,g \in \sO{\X}$.  Moreover, the Pfaffian $\ps^n \in \cohlgy[0]{\X,\der[2n]{\X}}$, which is a section of the anticanonical line bundle, gives a reduced defining equation for $\D$.  From this point of view, $\D$ is the \defn{degeneracy divisor of $\ps$}---the locus where its rank drops.  We will say that a Poisson structure $\ps$ is log symplectic if it is induced by a  log symplectic form in this way.

Conversely, given a generically nondegenerate Poisson structure $\ps$ with a reduced degeneracy divisor $\D$, it is always the case that $\ps \in \cohlgy[0]{\X,\der[2]{\X}(-\log\D)}$.  Indeed, $\D \subset \X$ is a Poisson subspace~(see \cite[Corollary 2.3]{Polishchuk1997} or \cite[Proposition 6]{Gualtieri2013a}), which implies that $\ps$ is tangent to $\D$.  The \defn{anchor map}
\[
\mapdef{\pss}{\forms[1]{\X}}{\der[1]{\X}}
	{\alpha}{\hook{\alpha}\ps}
\]
then extends to an isomorphism $\forms[1]{\X}(\log \D) \to \der[1]{\X}(-\log \D)$, so that $\pss$ may be inverted to obtain a log symplectic form $\omega$.  Notice that this argument works no matter how singular $\D$ is, thanks to the reflexivity of the sheaves involved.  We therefore have the basic
\begin{lemma}
Let $(\X,\D)$ be a log Calabi--Yau manifold.  Then generically symplectic Poisson structures on $\X$ with degeneracy divisor $\D$ are in canonical bijection with log symplectic forms on $(\X,\D)$.
\end{lemma}

Recall that for any Poisson bivector field $\ps$, the image of the anchor map $\pss : \forms[1]{\X} \to \der[1]{\X}$ is an involutive subsheaf of $\der[1]{\X}$ and gives rise to a foliation of $\X$ by even-dimensional symplectic leaves.  For a log symplectic structure on a connected complex manifold $\X$ of dimension $2n$, the open complement $\X \setminus \D$ is a symplectic leaf of dimension $2n$.  Meanwhile, the smooth locus of $\D$ is foliated by leaves of dimension $2n-2$; this foliation is defined by the kernel of the residue one-form $\Res \omega$, which is closed.  Then, the singular locus $\D_{sing}$ is foliated by leaves of dimension $2n-2$ or less.  Notice that having an open symplectic leaf is not enough to guarantee that $\ps$ is log symplectic, since the latter is possible even when the Pfaffian is not reduced.

We close this section with a description of the symplectic leaves of one of the elliptic log symplectic structures that will be the focus of \autoref{sec:ell-ex}.
\begin{example}\label{ex:e6tilde}
Let $(w,x,y,z)$ be coordinates on $\CC^4$ and let $\eta,\nu \in \CC$ be constants.  Then the Poisson brackets
\begin{align*}
\{w,x\} &= x &  \{y,z\} &= \eta\, x^2+\nu\, yz \\
\{w,y\} &= y & \{z,x\} &= \eta\, y^2+\nu\, zx \\
\{w,z\} &= z & \{x,y\} &= \eta\, z^2+\nu\, xy
\end{align*}
define a log symplectic structure whose degeneracy divisor $\D \subset \CC^4$ is given by the zeros of the cubic function
\[
f = \tfrac{1}{3}\eta(x^3+y^3+z^3) + \nu\, xyz.
\]
Thus, for generic values of $\eta$ and $\nu$, the hypersurface is a product $\D = \CC \times \D_0$, where $\D_0\subset \CC^3_{x,y,z}$ is the cone over an elliptic curve  $\E \subset \PP[2]$.  The singular locus $\D_{sing}\subset \D$ is a subscheme of multiplicity 8 supported on the line $x=y=z=0$. 

The complement $\CC^4\setminus \D$ is a symplectic leaf of dimension four, and the individual points of $\D_{sing}$ are zero-dimensional leaves.  The two-dimensional symplectic leaves give a regular foliation of the smooth locus $\D_{reg} = \D \setminus \D_{sing}$, which has the following description. Since $\D_{0}$ is the cone over $\E$ and $\D = \CC \times \D_0$,  we have a natural projection $\pi : \D_{reg} \to \CC \times \E$ whose fibres are copies of $\CC^\times$.  There is then a unique non-vanishing one-form $\alpha \in \cohlgy[0]{\E,\forms[1]{\E}}$ such that
\[
\Res \omega = \pi^*(p_1^* dw - p_2^* \alpha) \in \forms[1]{\D_{reg}},
\]
where $w$ is the coordinate on $\CC$ as above, and the maps $p_1 : \CC\times \E \to \CC$ and $p_2 : \CC\times \E \to \E$ are the projections.

Let $Z = \alpha^{-1} \in \cohlgy[0]{\E,\der[1]{\E}}$ be the vector field dual to $\alpha$.  Then $Z$ generates an action of the additive group $(\CC,+)$ on $\E$ by translations in the group law of the elliptic curve.  Combining this action with the standard translation action on $\CC$, we obtain the diagonal action of $(\CC,+)$ on the product $\CC\times \E$.  The symplectic leaves of $\D_{reg}$ are precisely the preimages of the orbits of this action under the projection $\pi : \D_{reg} \to \CC\times \E$. \qed
\end{example}

\subsection{Stability under deformations} \label{sec:irred}

For a compact complex manifold  $\X$, we denote by $\Pois{\X} \subset \cohlgy[0]{\X,\der[2]{\X}}$ the space of Poisson structures on $\X$.  It is the algebraic subvariety consisting of those sections $\ps$ that satisfy the integrability condition $[\ps,\ps] = 0 \in \cohlgy[0]{\X,\der[3]{\X}}$.  This condition amounts to a system of homogeneous quadratic equations on the finite dimensional vector space $\cohlgy[0]{\X,\der[2]{\X}}$.  The description of the irreducible components of $\Pois{\X}$ is a difficult problem and has important implications for noncommutative geometry.  We observe that some components of this variety are defined by log symplectic structures:
\begin{lemma}
If $\X$ is a compact complex manifold of dimension $2n$, then the set of log symplectic structures on $\X$ forms a (possibly empty) $\CC^\times$-invariant Zariski open subset  of $\Pois{\X}$.  Hence its closure is a union of irreducible components of $\Pois{\X}$.
\end{lemma}

\begin{proof}
Let $\E \subset \cohlgy[0]{\X,\der[2]{\X}}$ be the closed subset consisting of those sections $\ps$ such that $\ps^n=0$.  There is a natural map $\phi : \cohlgy[0]{\X,\der[2]{\X}} \setminus \E  \to \PP(\cohlgy[0]{\X,\der[2n]{\X}})$ given by $\ps \mapsto [\ps^n]$.  The projective space $\PP\rbrac{\cohlgy[0]{\X,\der[2n]{\X}}}$ parametrizes effective anticanonical divisors on $\X$, and the reduced divisors form a Zariski open subset $\U$.  Then the intersection of $\Pois{\X}$ with the open set $\phi^{-1}(\U)$ gives the set of log symplectic structures on $\X$.
\end{proof}

\section{Local structure of log symplectic manifolds}
\label{sec:local}

With the examples of the previous section in mind, we now develop some methods for studying the local and global structure of a log symplectic manifold with a singular degeneracy hypersurface.

\subsection{Gorenstein singular loci}
\label{sec:gorenstein}
The first step is to understand the simpler question of which hypersurfaces $\D \subset \X$ admit  a nondegenerate logarithmic two-form $\omega \in \forms[2]{\X}(\log\D)$---or, equivalently, a nondegenerate logarithmic biderivation $\ps \in \der[2]{\X}(-\log\D)$)---without worrying about the integrability condition $d \omega=0$.   Clearly, if $\D$ is a free divisor and $\dim \X = 2n$ is even, then such a logarithmic two-form exists locally.  Indeed, in a sufficiently small open set, we may simply take an $\sO{\X}$-module basis $\alpha_1,\ldots,\alpha_n,\beta_1,\ldots,\beta_n \in \forms[1]{\X}(\log \D)$ and set
\[
\omega = \alpha_1\wedge \beta_1 + \cdots + \alpha_n \wedge \beta_n.
\]
Of course, such a form will not, in general, be closed, and may not extend to a global logarithmic two-form.

Meanwhile, the log symplectic structure in \autoref{ex:e6tilde} has a degeneracy hypersurface whose singular locus has codimension three in the ambient space.  In light of the Jacobian criterion for freeness (\autoref{thm:jac-free}), such a hypersurface cannot be a free divisor.  Nevertheless, this criterion has an analogue in the skew-symmetric setting, which we now explain.

Let $(\X,\D)$ be a connected log Calabi--Yau manifold of dimension $2n$, and consider the anticanonical line bundle $\acan[\X] = \der[2n]{\X} \cong \sO{\X}(-\D)$. Let $\sA$ be its Atiyah algebroid, i.e.~the sheaf of first order differential operators on $\acan[\X]$.  Thus $\sA$ fits into an exact sequence
\begin{align}
\xymatrix{
0 \ar[r] & \sO{\X} \ar[r] & \sA \ar[r] & \der[1]{\X}\ar[r]  & 0, 
}\label{eqn:atiyah}
\end{align}
and splittings of this sequence of Lie algebroids are in bijective correspondence with flat connections on $\acan[\X]$.  Sections of $\ext[2]{\sA}$ can then be interpreted as skew-symmetric bidifferential operators $\acan[\X] \times \acan[\X] \to (\acan[\X])^{\otimes 2}$, and the second exterior power of \eqref{eqn:atiyah} gives a symbol map $\ext[2]{\sA} \to \der[2]{\X}$.

Given a nondegenerate logarithmic biderivation $\ps \in \cohlgy[0]{\X,\der[2]{\X}(-\log\D)}$ there is a canonical section 
\[
\ps_\sA \in \cohlgy[0]{\X,\ext[2]{\sA}}
\]
whose symbol is $\ps$.  This section can be described as follows: choose a local trivialization $\mu \in \acan[\X]$ and write $\ps^n = h\mu$ for a function $h \in \sO{\X}(-\D)$.  Consider the vector field
\begin{align}
Z = h^{-1}\hook{dh}\ps \in \der[1]{\X}(-\log \D). \label{eqn:Zdef}
\end{align}
Then the bidifferential operator $\ps_\sA$ may be defined by the formula
\[
\ps_\sA( f \mu , g \mu ) = \rbrac{\abrac{df\wedge dg,\ps}+gZ(f)-fZ(g)} \mu^2 \in (\acan[\X])^2
\]
for $f,g \in \sO{\X}$.  One can check that this definition is independent of the choice of local trivialization.

\begin{theorem}[\cite{Gualtieri2013a}]\label{thm:codim3}
With the notations above, the singular locus $\D_{sing} \subset \D$ is identified, as an analytic subspace of $\X$, with the vanishing locus of the top Pfaffian $\ps_\sA^n \in \cohlgy[0]{\X,\ext[2n]{\sA}}$.  As a result, the following statements hold:
\begin{enumerate}
\item Every irreducible component of $\D_{sing}$ has dimension $\ge 2n-3$.
\item The polynomial $c_1c_2-c_3\in \cohlgy[6]{\X,\ZZ}$ in the Chern classes of $\X$ is supported on $\D_{sing}$.
\item If $\dim \D_{sing} = 2n-3$, then $\D_{sing}$ is Gorenstein with dualizing sheaf $\acan[\X]|_{\D_{sing}}$.  Moreover, there is a canonical locally free resolution of $\sO{\D_{sing}}$ having the form
\[
\xymatrix{
0 \ar[r] & \can[\X]^2 \ar[r]^-{\ps_\sA^n} & \can[\X]\otimes \sA^\vee \ar[r]^-{\pss_\sA} & \can[\X]\otimes \sA \ar[r]^-{\ps_\sA^n} & \sO{\X} \ar[r] & \sO{\D_{sing}} \ar[r] & 0
}
\]
If, in addition, $\X$ is compact, then the fundamental class of the singular locus is $[\D_{sing}] = c_1c_2-c_3 $.
\end{enumerate}
\end{theorem}

\begin{proof}This theorem was proved in \cite{Gualtieri2013a} under the assumption that $\ps$ was a Poisson structure, but the integrability condition is not actually used in the proof.  So, we shall omit most of the proof here, and simply recall why the Pfaffians define the singular locus.

A local trivialization $\mu \in \acan[\X]$ gives a splitting $\sA \cong \der[1]{\X}\oplus\sO{\X}$ of \eqref{eqn:atiyah}, and hence a decomposition $\ext[2n]{\sA} \cong \der[2n]{\X} \oplus \der[2n-1]{\X}$.  With respect to this decomposition, we may write $\ps_\sA^n = (\ps^n,n Z\wedge \ps^{n-1})$ with $Z$ as defined in \eqref{eqn:Zdef}.  But
\begin{align*}
Z \wedge \ps^{n-1} &=  (h^{-1}\hook{dh}\ps)\wedge \ps^{n-1} \\
&= \tfrac{1}{n}h^{-1}\hook{dh}\ps^n\\
&= \tfrac{1}{n}\hook{dh}\mu
\end{align*}
and hence $\ps_{\sA}^n = (h\mu,\hook{dh}\mu)$.  Since $\mu$ is nonvanishing, the zero locus of $\ps_{\sA}^n$ coincides with the simultaneous vanishing locus of $h$ and $dh$, which is precisely the singular locus of $\D$, as claimed.
\end{proof}

Locally, this theorem has a partial converse:
\begin{proposition}\label{prop:codim3-converse}
Let $\D \subset (\CC^{2n},0)$ be a reduced hypersurface germ.  Suppose that the singular locus of $\D$ is Gorenstein of pure dimension $2n-3$.  Then there exists a nondegenerate logarithmic biderivation $\ps \in \der[2]{\CC^2,0}(-\log \D)$.
\end{proposition}
\begin{proof}
Let $\sO{} = \sO{\CC^{2n,0}}$ be the ring of germs of analytic functions and denote by $\der[1]{} = \der[1]{\CC^{2n},0}$ be the $\sO{}$-module of vector fields.  
Let $h \in \sO{}$ be a defining equation for $\D$, and let $\sO{\D_{sing}}$ be the ring of functions on the singular locus.  Setting $\sA = \sO{} \oplus \der[1]{}$, the map
$(h,dh): \sA \to \sO{}$ has cokernel $\sO{\D_{sing}}$, giving the beginning of a free resolution of $\sO{\D_{sing}}$.  Its kernel is identified with $\der[1]{}(-\log\D)$ by the projection $\sA \to \der[1]{}$.

On the other hand, by the Buchsbaum--Eisenbud structure theorem~\cite{Buchsbaum1977} for codimension three Gorenstein ideals, there exist an integer $k >0$, a free module $\sE$ of rank $2k+1$, and a skew-symmetric form $\rho\in\ext[2]{\sE}$, such that the minimal free resolution of $\sO{\D_{sing}}$ has the form
\[
\xymatrix{
0 \ar[r] & (\det{\sE^\vee})^2 \ar[r]^-{\rho^k} & \sE^\vee \ar[r]\otimes \det{\sE^\vee} \ar[r]^{\rho \otimes 1} & \sE \otimes \det{\sE^\vee} \ar[r]^-{\rho^k} & \sO{}
}
\]
where the outer maps are the Pfaffians $\rho^{k} \in \ext[2k]{\sE} \cong \sE^\vee \otimes \det \sE$.  Therefore, by minimality, there exists a free module $\sF$ equipped with an isomorphism $\sA \cong (\sE \otimes \det{\sE^\vee}) \oplus \sF$ that identifies the submodule $\der[1]{}(-\log \D) \subset \sA$ with $\image \rho \oplus \sF$.

Since the ranks of $\sA$ and $\sE$ are odd, the rank of $\sF$ must be even.  Hence we may choose a nondegenerate skew form $\rho' \in \ext[2]{\sF}$.  Under the isomorphism $\sA \cong \sE\oplus \sF$, the sum $\ps_\sA = \rho + \rho'$ defines an element of $\ext[2]{\sA}$, and we obtain a free resolution
\[
\xymatrix{
0 \ar[r] & \sO{} \ar[r]^{(h,dh)^t} & \sA^\vee \ar[r]^{\ps_\sA} & \sA \ar[r]^{(h,dh)} & \sO{}
}
\]
of $\sO{\D_{sing}}$.  Let $\ps $ be the image of $\ps_{\sA}$ under the projection $\ext[2]{\sA} \to \der[2]{}$.  The fact that the sequence above is a complex implies that $\ps$ lies in $\der[2]{\X}(-\log\D)$.  Moreover, the map $(h,dh)$ is identified with the Pfaffian $\ps_\sA^n$ under the isomorphism $\ext[2n]{\sA} \cong \der[2n]{} \oplus \der[2n-1]{} \cong \sO{} \oplus \forms[1]{}$ given by an appropriate choice of volume form.  Hence $h^{-1}\ps^n$ trivializes $\der[2n]{}$, so that $\ps$ is nondegenerate.
\end{proof}

\begin{example}
Let $f \in \CC[x,y,z]$ be a quasi-homogeneous polynomial with weights $(a,b,c) \in \ZZ_{>0}^3$ and suppose that $0$ is the only critical point of $f$.  Let $\D_0 \subset \CC^3$ be the zero locus of $f$.  Because of the quasihomogeneity, the singular locus of the product $\D = \CC \times \D_0 \subset \CC^4$ is the complete intersection defined by the equations $\cvf{x}f=\cvf{y}f=\cvf{z}f = 0$, and hence it is Gorenstein.

Correspondingly, $\D$ supports a nondegenerate logarithmic biderivation, an example of which may be constructed as follows.  Let $E = ax\cvf{x}+by\cvf{y}+cz\cvf{z}$ be the weighted Euler vector field, and extend the coordinates $(x,y,z)$ on $\CC^3$ to a coordinate system $(w,x,y,z)$ on $\CC^4$.  Then one readily computes that
\[
\ps = E \wedge \cvf{w} + \hook{df} (\cvf{x}\wedge \cvf{y}\wedge\cvf{z}) \in \der[2]{\CC^4}(-\log \D)  
\]
is a nondegenerate logarithmic biderivation.  However, it will not, in general be integrable.  In fact, we will see in \autoref{sec:ell-ex} that $\D$ supports a log symplectic form if and only if the degree of $f$ is equal to $a+b+c$. \qed
\end{example}


\subsection{The local Moser trick}
\label{sec:normal}

We now return to the integrable case, in which the nondegenerate logarithmic form is closed.  Clearly if $(\X,\omega)$ and $(\X',\omega')$ are log symplectic structures that are isomorphic in neighbourhoods of points $p \in \X$ and $p' \in \X'$, then their degeneracy hypersurfaces $\D$ and $\D'$ are also isomorphic in the same neighbourhoods.  In other words, the singularity type of $\D$ at $p$ is a local invariant of the log symplectic structure.  In this section, we study log symplectic structures for which the underlying singularity type is fixed, and find local analogues of several standard cohomological properties of \emph{compact} symplectic manifolds, and log symplectic manifolds with smooth degeneracy hypersurfaces~\cite{Guillemin2014}.

For a reduced hypersurface germ $\D\subset (\X,p)$, let $\D=\D_1\cup\cdots\cup\D_k$ be a decomposition into irreducible components, and let $h_1,\ldots,h_k$ be defining equations for the components.  If the $\sO{\X,p}$-module $\forms[1]{\X}(\log \D)_p$ is generated by the meromorphic forms $f_1^{-1}df_1,\ldots,f_k^{-1}f_k$ and the holomorphic forms $\forms[1]{\X}$, we will say that $\forms[1]{\X}(\log \D)_p$ is \defn{generated by logarithmic differentials}.  Saito has shown that $\forms[1]{\X}(\log\D)_p$ is generated by logarithmic differentials if and only if each component $\D_i$ is normal and the intersections of the components are sufficiently transverse; see \cite[Theorem 2.9]{Saito1980} for the exact formulation.  These conditions are satisfied, in particular, if $\D$ is smooth, normal, or has normal crossings singularities.

Our results pertain to the \defn{local logarithmic de Rham cohomology} at a point $p \in \D_{sing}$, which is the stalk cohomology $\cohlgy{\X,\log\D}[dR,p]$ of the complex of sheaves $(\forms{\X}(\log \D),d)$ at $p$. 

\begin{proposition}
Let $(\X,\D,\omega)$ be a log symplectic manifold of dimension $2n$, and let $p \in \D_{sing}$ be a singular point at which $\D$ is normal. (Hence $\dim \D_{sing} = 2n-3$ at $p$ by \autoref{thm:codim3}.)  Then the local logarithmic cohomology class
\[
[\omega] \in \cohlgy[2]{\X,\log \D}[dR,p]
\]
is nonzero.
\end{proposition}

\begin{proof}
Suppose, to the contrary, that $[\omega] = 0$.  Then
\[
\omega = d \alpha
\]
for some $\alpha \in \forms[1]{\X}(\log \D)_p$.  Let $f$ be a reduced defining equation for $\D$ near $p$.  Since $\D$ is normal, it is irreducible, so that $\forms[1]{\X}(\log \D)_p$ is generated by the logarithmic differential $f^{-1}df$.   We may therefore write
\[
\alpha = g f^{-1} df + \beta
\]
with $g \in \sO{\X,p}$ and $\beta \in \forms[1]{\X,p}$, so that
\[
\omega = f^{-1} dg \wedge df + d \beta.
\]
Therefore
\[
\omega^n = nf^{-1} \, dg \wedge df \wedge d \beta^{n-1} + (d\beta)^n.
\]
Viewed as a section of $\forms[2n]{\X}(\D)$, the $2n$-form $(d \beta)^n $ vanishes at $p$.  Nondegeneracy of $\omega$ then implies that the section
\[
f^{-1} \, dg \wedge df \wedge d \beta^{n-1}
\]
is a local trivialization of $\forms[2n]{\X}(\D)$ near $p$, which in turn implies that $df(p) \ne 0$.  But $p$ is a singular point of $\D$, and hence $df(p)=0$, a contradiction.
\end{proof}

\begin{remark}
The proposition fails, in general, if $\D$ is not normal.  Indeed, the logarithmic cotangent bundle of any free divisor carries an exact log symplectic form, defined in the same way as for usual cotangent bundles.  \qed
\end{remark}

\begin{lemma}[Local Moser trick]\label{lem:local-moser}
Let $\D \subset \X$ be a reduced hypersurface and let $p \in \D_{sing}$ be a point such that $\forms[1]{\X}(\log\D)_p$ is generated by logarithmic differentials.  Suppose that $\omega_t \in \forms[2]{\X}(\log \D)_p$, $t \in [0,1]$ is a smooth family of germs of log symplectic forms such that the local logarithmic cohomology class
\[
[\omega_t] \in \cohlgy[2]{\X,\log\D}[dR,p]
\]
is independent of $t$.  Then there exists a family of germs of automorphisms $\phi_t \in \Aut{\X,\D}_p$ such that $\phi^*_t\omega_t = \omega_0$.
\end{lemma}

\begin{proof}

By assumption, the logarithmic two-form
\[
\dot\omega_t = \deriv{\omega_t}{t}
\]
is exact.  We may therefore choose a smooth family $\alpha_t \in \forms[1]{\X}(\D)_p$ of logarithmic one-forms such that $d\alpha_t = \dot \omega_t$.  Let $f_1,\ldots,f_k$ be defining equations for the irreducible components of $\D$.  Because $\forms[1]{\X}(\log\D)_p$ is generated by logarithmic differentials, we may write
\[
\alpha_t = \sum_{i=1}^k g_{i,t} f^{-1}_i df_i + \beta_t
\]
where $g_{i,t} \in \sO{\X,p}$ and $\beta_{t} \in \forms[1]{\X,p}$.  Replacing $g_{i,t}$ with $g_{i,t}-g_{i,t}(p)$, and changing $\beta_t$ by an exact form depending on $t$, we may assume that $g_{i,t}(p) = 0$ and $\beta_{i,t}(p) = 0$ for all $t$.  Hence $\alpha_t$ vanishes at $p$ when viewed as a section of $\forms[1]{\X}(\log \D)$.

By nondegeneracy, there is a unique time-dependent logarithmic vector field $Z_t \in \der[1]{\X}(-\log\D)_p$ such that
\[
\hook{Z_t}\omega_t = \alpha_t,
\]
and since $\alpha_t$ vanishes at $p$, the holomorphic vector field $Z_t$ also vanishes at $p$ for all $t$.  Therefore, $Z_t$ integrates to a family of germs of automorphisms $\phi_t \in \Aut{\X,p}$, and since $Z_t$ is logarithmic, these automorphism preserve the subgerm $(\D,p) \subset (\X,p)$.

We now apply the standard calculation
\begin{align*}
\deriv{\phi_t^* \omega_t}{t} = \lie{Z_t}\omega_t 
=\hook{Z_t}d\omega_t + d\hook{Z_t}\omega_t 
= d \alpha_t
= \dot \omega_t
\end{align*}
which implies that that $\phi_1^*\omega_1 = \omega_0$, as required.
\end{proof}

This result implies that the local logarithmic cohomology class determines the germ of the log symplectic form up to isomorphism:
\begin{theorem}\label{thm:germ-moduli}
Let $\D\subset \X$ be a reduced hypersurface and let $p \in \D$ be a point such that $\forms[1]{\X}(\log\D)_p$ is generated by logarithmic differentials.  Suppose that $\omega_0,\omega_1 \in \forms[2]{\X}(\log\D)_p$ are two germs of log symplectic structures with the same local logarithmic cohomology classes
\[
[\omega_0] = [\omega_1] \in \cohlgy[2]{\X,\log \D}[dR,p].
\]
Then $\omega_0$ and $\omega_1$ are isotopic.
\end{theorem}

\begin{proof}
For $\lambda \in \CC$, we define the logarithmic two-form
\[
\omega_\lambda = (1-\lambda) \omega_0 + \lambda \omega_1 .
\]
Clearly the class $[\omega_\lambda] \in \cohlgy[2]{\X,\log \D}[p]$ is independent of $\lambda$.

Let $\bL = \forms[2n]{\X}(\D)|_{p}$ be the fibre of the logarithmic canonical line bundle at $p$, and consider the map
\[
\mapdef{\Pf}{\CC}{\bL}
{\lambda}{(\omega_\lambda)^n(p)}
\]
Then $\omega_\lambda$ is nondegenerate in a neighbourhood of $p$ if and only if $\Pf(\lambda) \ne 0$.  In particular, $\Pf(0)$ and $\Pf(1)$ are both nonzero.  Since $\Pf(\lambda)$ depends polynomially on $\lambda$, we may choose a smooth path $\gamma : [0,1] \to \CC$ such that $\gamma(0) = 0$, $\gamma(1) = 1$ and $\Pf(\gamma(t)) \ne 0$ for all $t$.  But then $\omega_{\gamma(t)}$ for $t \in [0,1]$ gives a smoothly varying family of log symplectic forms whose local cohomology classes are all the same.  We are therefore in the situation of \autoref{lem:local-moser} and conclude that the germs of $\omega$ and $\omega'$ are isomorphic.
\end{proof}

The main application of this result is to give streamlined proofs of local normal forms for log symplectic structures.  We give now two examples; the approach will be used again when we discuss elliptic structures in \autoref{sec:ell-ex}.

\begin{example}
Let $p\in \D $ be a smooth point.  Then $\forms[1]{\X}(\log\D)_p$ is generated by logarithmic differentials and $\cohlgy[2]{\X,\log\D}[dR,p] = 0$.  Thus, in the neighbourhood of a smooth point $p \in \D$, all log symplectic forms are isomorphic, as we know from the log Darboux normal form~\eqref{eqn:log-Darboux}.\qed
\end{example}

\begin{example}
Let $p \in \D$ be a point at which $\D$ may be decomposed locally into $k$ components with normal crossings.  Choose coordinates $(x_1,\ldots,x_k,y_1,\ldots,y_m)$ so that $x_1x_2\cdots x_k$ is a local defining equation for $\D$.  Then the first logarithmic cohomology is a $k$-dimensional vector space
\[
\V = \cohlgy[1]{\X,\log \D}[dR,p] = \CC \cdot \abrac{[x_1^{-1}dx],\ldots,[x_k^{-1}dx_k]}
\] 
and the wedge product identifies $\cohlgy{\X,\log \D}[dR,p] \cong \ext{\V}$ as rings.  

In particular, for $k \ge 2$, every logarithmic two-form is cohomologous to one of the form
\[
\omega_0 = \sum_{1 \le i<j \le k} a_{ij} \frac{dx_i}{x_i}\wedge \frac{dx_j}{x_j}
\]
where $A = (a_{ij})$ is a skew symmetric matrix of constants. The $(i,j)$-entry can be interpreted as a period: it is the integral of $\omega$ over a small real torus $\T_{ij} \cong \bS^1\times \bS^1$ that encircles the intersection $\D_i\cap \D_j$ in a neighbourhood of $p$.

If $k = \dim \X$, then the above calculation of the cohomology shows that any nonvanishing logarithmic volume form is nontrivial in cohomology.  Now $\omega_0^n = \Pf(A)\frac{dx_1}{x_1}\wedge \cdots \frac{dx_n}{x_n}$, and hence a form cohomologous to $\omega_0^n$ can only be log symplectic if the Pfaffian $\Pf(A) \ne 0$, i.e.~the matrix $A$ is nonsingular.  In this case $\omega_0^n$ is already log symplectic.  Hence every log symplectic structure on $\D$ is locally isomorphic to one of this form, where the nonsingular matrix $A$ is determined by the periods of the log symplectic form.

On the other hand, if $k < \dim \X$ then $\omega_0$ is degenerate.  In order to obtain a log symplectic form, we must extend $\omega_0$ by adding terms involving the $y$ coordinates, to obtain a cohomologous form:
\[
\omega = \sum_{1 \le i < j \le k} a_{ij} \frac{dx_i}{x_i}\wedge \frac{dx_j}{x_j} + \sum_{i=1}^k\sum_{j=1}^m b_{ij} \frac{dx_i}{x_i}\wedge dy_j + \sum_{1 \le i < j \le m} c_{ij} dy_i \wedge dy_j
\]
for a $k\times m$ matrix of constants $B= (b_{ij})$ and an $m\times m$ skew-symmetric matrix $C = (c_{ij})$.  Then $\omega$ will be log symplectic if and only if the block matrix
\[
\begin{pmatrix}
A & B \\
-B & C
\end{pmatrix}
\]
of coefficients is invertible.  Two such log symplectic forms will be isotopic near $p$ if and only if they share the same matrix $A$ of periods.  The matrices $B$ and $C$ do not affect the isomorphism class and thus they can be further simplified, but the optimal simplification depends on $k$ and $A$. \qed
\end{example}

\section{Elliptic log symplectic manifolds}

\label{sec:ell-ex}

\subsection{The elliptic normal form theorem}

Let $(\X,\D,\omega)$ be a log symplectic manifold and let $\ps$ be the corresponding Poisson bivector.  Given a volume form $\mu \in \can[\X]$, we have $\ps^n = f \mu^{-1}$, where $f \in \sO{\X}(-\D)$ is a local defining equation for $\D$.  We may then consider the log Hamiltonian vector field
\[
Z_\mu = f^{-1}\hook{df} \ps \in \der[1]{\X}(-\log \D).
\]
This vector field is nothing but the \defn{modular vector field of $\ps$} with respect to the volume form $\mu$, as introduced by Weinstein~\cite{Weinstein1997} for general Poisson manifolds.  One can easily verify that $Z_\mu$ is $\mu$-divergence free, i.e.~$\lie{Z_\mu}\mu = 0$.

If we change the volume form, then $Z_\mu$ changes by a Hamiltonian vector field.  Hence, if $p \in \X$ is a point where the Poisson structure vanishes, the value  of $Z_\mu$ at $p$ is independent of the equation $f$.  We therefore have a natural section
\[
Z \in \cohlgy[0]{\Y,\der[1]{\X}|_\Y}
\]
where $\Y \subset \X$ is the vanishing locus of $\ps$.  According to \cite[Theorem 19]{Gualtieri2013a}, this section may alternatively be computed as follows: restricting the one-jet of $\ps$ to $\Y$, we obtain a natural section $j^1\ps|_\Y \in \forms[1]{\X}\otimes\der[2]{\X}|_\Y$, and $Z$ is the image of this section under the interior contraction $\forms[1]{\X} \otimes \der[2]{\X} \to \der[1]{\X}$.  Hence, $Z$ is a linear combination of first derivatives of the Poisson structure along its zero locus.

\begin{remark}\label{rem:modres}
Since $Z_f$ is log Hamiltonian, it is a symmetry of $\ps$, and hence it is tangent $\Y$, defining a vector field on $\Y$.  More generally, one can show~\cite{Gualtieri2013a} that on the locus $\Dgn[2k]{\ps} \subset \X$ where a Poisson structure $\ps$ has rank at most $2k$, there is a natural residue of $\ps$, which is a global section of $\der[2k+1]{\Dgn[2k]{\ps}}$.\qed
\end{remark}

Suppose now that $p \in \D_{sing}$ is a singular point at which $\D$ is normal.  By \autoref{thm:codim3}, the codimension of $\D_{sing}$  in $\X$ is equal to three. Since $\D_{sing}$ is a Poisson subspace (\cite[Corollary 2.4]{Polishchuk1997} or \cite[Lemma 2.3]{Gualtieri2013a}), $\D_{sing}$ contains the symplectic leaf $\bL \subset \X$ through $p$, and hence $\bL$, being even-dimensional, must have codimension at least four.  The rest of the paper is concerned with singularities for which the following transversality condition is satisfied:
\begin{definition}\label{def:ell-point}
Let $(\X,\D,\omega)$ be a log symplectic manifold, let $p \in \D_{sing}$ be a singular point and let $\bL \subset \X$ be the symplectic leaf through $p$.  We say that $p$ is an \defn{elliptic point of $\omega$} if the following conditions hold:
\begin{enumerate}[noitemsep]
\item $\D$ is normal at $p$,
\item $\bL$ has codimension four in $\X$, and
\item The modular vector field $Z$ is transverse to $\bL$ at $p$.
\end{enumerate}
We say that $(\X,\omega)$ is \defn{purely elliptic} if every singular point of its degeneracy divisor is an elliptic point.
\end{definition}

\begin{remark}\label{rem:ell-pfaff}
Equivalently, $p \in \D_{sing}$ is elliptic if $\D$ is normal at $p$ and the rank of $\ps_{\sA}(p)$ is equal to $\dim \X - 2$, where $\ps_{\sA}$ is the tensor defined in \autoref{sec:gorenstein}.\qed
\end{remark}

\begin{remark}
When $\dim \X = 4$, the second condition in the definition is redundant, and the third is simply the requirement that $Z$ be nonzero at $p$.\qed
\end{remark}

Using Weinstein's splitting theorem~\cite{Weinstein1983}, we see that in a sufficiently small neighbourhood of an elliptic point, the log symplectic structure decomposes as a  product of a symplectic manifold of dimension $\dim \X-4$, and a four-dimensional purely elliptic structure.  As a result, the local structure near an elliptic point is completely determined by the following
\begin{theorem}\label{thm:local-elliptic}
Let $(\X,\D,\omega)$ be a log symplectic manifold of dimension four, and let $p \in \D_{sing}$ be an elliptic point.  Then there exist coordinates $(w,x,y,z)$ on $\X$ centred at $p$ in which the Poisson brackets have the form
\begin{align}
\{w,x\} &= ax & \{x,y\} &= \lambda\cvf{z}f \nonumber \\ 
\{w,y\} &= by & \{y,z\} &= \lambda \cvf{x}f \label{eqn:elliptic-C4}\\
\{w,z\} &= cz & \{z,x\} &= \lambda \cvf{y}f\nonumber
\end{align}
for some positive integers $(a,b,c) \in \ZZ_{>0}^3$ and polynomial $f$ appearing in \autoref{tab:e678}, together with a constant $\lambda \in \CC^\times$.

In particular, the zero locus of the Poisson structure, defined by the equations $x=y=z=0$, is smooth at $p$; and the degeneracy hypersurface, defined by the equation $f=0$, is locally the product of a smooth curve and a simple elliptic surface singularity.
\end{theorem}

\begin{table}[h]
\begin{center}
\caption{The simple elliptic surface singularities, parametrized by $\tau \in \CC^\times$}\label{tab:e678}
\begin{tabular}{c|c|c|c}
Type & $(a,b,c)$ & Quasi-homogeneous polynomial $f$ & Milnor number \\
\hline
$\tE_{6,\tau}$ & $(1,1,1)$ &  $x^3+y^3+z^3+ \tau xyz$ & 8\\
$\tE_{7,\tau}$ & $(1,1,2)$ & $x^4+y^4+z^2 + \tau xyz$ & 9\\
$\tE_{8,\tau}$ & $(1,2,3)$ & $x^6 + y^3 + z^2 + \tau xyz$ & 10
\end{tabular}
\end{center}
\end{table}

The rest of this section is devoted to the proof of \autoref{thm:local-elliptic}.  The first step is the following
\begin{lemma}
There exist coordinates $(w,x,y,z)$ on $\X$ centred at $p$ and a quasi-homogeneous polynomial $f \in \CC[x,y,z]$ with an isolated critical point at $0$ such that $f$ gives a reduced defining equation for $\D$.
\end{lemma}

\begin{proof}
Let $Z$ be the modular vector field with respect to some volume form $\mu \in \forms[4]{\X}$ defined near $p$.  Since $Z(p) \ne 0$ and $\lie{Z}\mu=0$, we may choose a coordinate system $(w,x,y,z)$ near $p$ such that $Z = \cvf{w}$ and $\mu = dw\wedge dx\wedge dy\wedge dz$.  Since $\lie{Z} \ps = 0$, one can easily compute that the germ of the Poisson tensor $\ps$ has the form
\[
\ps = E \wedge \cvf{w} + \ps_0
\]
where $\ps_0$ is the germ of a Poisson structure on $\CC^3$, with coordinates $(x,y,z)$, and $E$ is the germ of a vector field on $\CC^3$ satisfying $\lie{E} \ps_0 = 0$.  In particular,
\[
\ps \wedge \ps = f \mu^{-1},
\]
where $f$ is a function depending only on $(x,y,z)$, so that degeneracy hypersurface $\D$ is a product of a smooth curve with the surface $\D_0 = f^{-1}(0) \subset \CC^3_{x,y,z}$.  Since we are assuming that $\dim_p \D_{sing} = 1$, the surface $\D_0$ must have an isolated singularity at the origin.

Using the decomposition, we compute the modular vector field
\[
Z = f^{-1}\hook{df}\ps = f^{-1}(\lie{E} f) \cvf{w} + f^{-1}\hook{df}\ps_0.
\]
Since $Z = \cvf{w}$ and $\ps_0$ has no components involving $\cvf{w}$, this equation implies that $\lie{E}f = f$.  The lemma now follows from a theorem of K.~Saito~\cite{Saito1971}.
\end{proof}

Using the previous lemma and the cohomological parametrization of local normal forms (\autoref{thm:germ-moduli}), the proof of \autoref{thm:local-elliptic} reduces to the following
\begin{proposition}
Let $(w,x,y,z)$ be coordinates on $\CC^4$, let $(a,b,c) \in \ZZ_{>0}^3$ be positive weights, and let $f \in \CC[x,y,z]$ be a polynomial that is quasi-homogeneous of degree $k$ with respect to these weights.  Let $\D \subset \CC^4$ be the hypersurface defined by the vanishing of $f$.  Then the following statements hold:
\begin{enumerate}
\item If $\D$ supports a log symplectic structure, then $k = a+b+c$.  Hence $f$ is equal, after a change of coordinates, to one of the polynomials in \autoref{tab:e678}.
\item When $k=a+b+c$, the log symplectic forms associated with the Poisson brackets~\eqref{eqn:elliptic-C4} represent all nonzero cohomology classes in $\cohlgy[2]{\CC^4,\log\D}[dR,p]$.
\end{enumerate}
\end{proposition}

\begin{proof}
Let $\D_0 \subset \CC^3$ be the intersection of $\D$ with the hyperplane $w=0$, so that $\D = \CC\times \D_0$.  By the K\"unneth theorem for logarithmic de Rham cohomology~\cite[Lemma 2.2]{Castro-Jimenez1996}, the pullback along the projection $(\CC^4,\D)\to (\CC^3,\D_0)$ provides an isomorphism
\[
\cohlgy[2]{\CC^4,\log \D}[dR,0] \cong \cohlgy[2]{\CC^3,\log\D_0}[dR,0]
\]
Because of the holomorphic Poincar\'e lemma, the complex $\forms{\CC^3}$ is exact in positive degrees.  Hence the residue exact sequence~\eqref{eqn:res} gives an isomorphism
\[
\cohlgy[2]{\CC^3,\log\D_0}[dR,0] \cong \cohlgy[1]{\Rforms{\D_0}}[0]
\]
with the cohomology of the regular differential forms on $\D_0$.  But since $\D_0$ is normal and $\Rforms{\D_0}$ is reflexive, we have $\Rforms{\D_0} \cong (\forms{\D_0})^{\vee\vee} = \sHom{\der{\D_0},\sO{\D_0}}$.  The cohomology of this complex was computed in~\cite[Theorem 5.2]{Eriksen2009}.  From that result, we see that every cohomology class in $\cohlgy[2]{\CC^3,\log\D_0}[0]$ is represented by a unique element of the form
\[
\omega_h = f^{-1}h\, \hook{E}(dx\wedge dy\wedge dz)
\]
where $h \in \CC[x,y,z]$ is quasi-homogeneous of degree $k-a-b-c$.

Hence every closed logarithmic form $\omega \in \forms[2]{\CC^4}(\log\D)$ may be written as
\[
\omega = f^{-1}h \, \hook{E}(dx\wedge dy\wedge dz) + d\alpha
\]
where $\alpha \in \forms[1]{\CC^4}(\log\D)$.  Since $\D$ is normal, $\forms[1]{\CC^4}(\log\D)$ is generated by logarithmic differentials, so that $\alpha = g f^{-1}df + \beta $ with $g \in \sO{\CC^4}$ and $\beta\in \forms[1]{\CC^4}$ holomorphic.  

We must determine when such an $\omega$ is nondegenerate.  Using the identity $\lie{E}f = k f$ and the fact that $f$ and $df$ vanish at the origin, one easily computes
\begin{align*}
\omega^2 &= 2k hf^{-1}\, dx\wedge dy \wedge dz \wedge dg  \mod (x,y,z) \cdot \forms[4]{\CC^4}(\D).
\end{align*}
Hence $\omega$ can only be nondegenerate if the top degree form $hdx\wedge dy \wedge dz\wedge dg$ is nonvanishing near the origin.   This condition forces the quasi-homogeneous polynomial $h$ to be a nonzero constant.  Since its degree is equal to $k-a-b-c$, we must have $k=a+b+c$.  The first statement now follows from K.~Saito's classification of such quasi-homogeneous polynomials~\cite{Saito1971}; see also the formulae in~\cite[Proposition 2.3.2]{Etingof2010}.

For the second statement, we note that, since the choices of $g$ and $\beta$ do not affect the cohomology class of $\omega$, we may as well take $g = k^{-1}w$ and $\beta=0$.  Thus every log symplectic form on $(\CC^4,\D)$ is cohomologous at $0$ to the form
\[
\omega = (k\lambda f)^{-1}\, \hook{E}(dx\wedge dy \wedge dz) + k^{-1} f^{-1} \, df \wedge dw
\]
for a unique constant $\lambda \in \CC^\times$.  Inverting $\omega$, we obtain the Poisson brackets of \eqref{eqn:elliptic-C4}, as required.
\end{proof}

\subsection{Basic properties of purely elliptic structures}

\subsubsection{Irreducible components of $\Pois{\X}$}
Recall from \autoref{sec:irred} that the class of log symplectic structures on a compact complex manifold $\X$ is stable under deformations, giving irreducible components in the space $\Pois{\X}$ of Poisson structures on $\X$.  The same is true for the more restricted class of elliptic structures:
\begin{proposition}\label{prop:ell-stable}
Let $\X$ be a compact complex of dimension $2n$.  Then the set of purely elliptic log symplectic structures on $\X$ gives a Zariski open subset of $\Pois{\X}$, and hence its closure is a union of irreducible components.
\end{proposition}

\begin{proof}
According to \autoref{rem:ell-pfaff}, ellipticity is equivalent to requiring that $\D$ be normal and the Pfaffian $\ps_{\sA}^{n-1}$ be nonvanishing, where $\ps_{\sA}$ is defined as in \autoref{sec:gorenstein}.  Both of these conditions are open in the Zariski topology on $\cohlgy[0]{\X,\der[2]{\X}}$, as required.
\end{proof}

\subsubsection{The canonical bundle of the singular locus}

Let $(\X,\D,\omega)$ be a purely elliptic log symplectic manifold of dimension $2n$ and let $\ps$ be the corresponding Poisson structure.  According to \autoref{thm:local-elliptic}, the locus where $\ps$ has rank $2n-4$ is a manifold, and is equal to the reduced space $\Y = (\D_{sing})_{red} \subset \X$ underlying the singular locus of $\D$.  It therefore carries a regular Poisson structure whose symplectic leaves have codimension one in $\Y$.   As mentioned in \autoref{rem:modres}, $\Y$ also carries a natural top-degree polyvector field, which is locally given by the formula $\ps^{n-2}\wedge Z|_\Y$ where $Z$ is the modular vector field in a local trivialization of $\can[\X]$.  The ellipticity condition ensures that this tensor is nonvanishing.  We therefore have the
\begin{lemma}
The canonical bundle of $\Y$ is trivial.
\end{lemma}
If $\X$ is projective, this puts strong constraints on $\Y$: when $\dim \X = 4$, each connected component of $\Y$ is an elliptic curve.  Meanwhile, when $\dim \X = 6$, the manifolds $\Y$ give part of the classification~\cite{Druel1999}.

\subsubsection{Topological constraints}

According to the local normal form, we may decompose $\Y$ as a disjoint union
\[
\Y = \Y_6 \sqcup \Y_7 \sqcup \Y_8
\]
of open submanifolds, where $\Y_i$ denotes the set of points at which $\D$ has singularities of type $\tE_{i}$. 

Considering the Milnor numbers of the elliptic singularities in~\autoref{tab:e678}, we see that the length of $\D_{sing}$ along $\Y_{i}$ is equal to $i+2$.  Hence we have the following formula for the fundamental class in singular homology:
\[
[\D_{sing}] = 8[\Y_6] + 9[\Y_7]+10[\Y_8] \in \hlgy[\dim\X-6]{\X,\ZZ}.
\]
Combining this result with \autoref{thm:fano-singular}, we obtain the
\begin{proposition}\label{prop:678charclass}
For a purely elliptic log symplectic manifold $(\X,\D,\omega)$, we have the equality
\[
8[\Y_6] + 9[\Y_7] + 10[\Y_8] = (c_1c_2-c_3)\cap [\X] \in \hlgy[\dim \X - 6]{\X,\ZZ},
\]
where $\Y_i \subset \X$ is the locus where $\D$ has singularities of type $\tE_i$.
\end{proposition}

\section{Feigin--Odesskii structures and Fano fourfolds}

In~\cite{Feigin1989}, Feigin and Odesskii introduced a remarkable collection of Poisson structures on projective space that are associated to a given elliptic curve $\E$. In these examples, the projective space is interpreted~\cite{Feigin1998,Polishchuk1998} as a moduli space parametrizing certain configurations of vector bundle maps over $\E$.

Among this collection are the families $q_{d,1}$ of Poisson structures on $\PP[d-1]$, parametrized up to rescaling by the choice of an elliptic normal curve $\E \subset \PP[d-1]$.  When $d = 2n+1$ is odd, these Poisson structures on $\PP[2n]$ are log symplectic; the degeneracy hypersurface $\D \subset \PP[2n]$ is the $(n-1)$-secant variety $\Sec[n-1]{\E}$ of the elliptic curve $\E$, i.e.~the closure of the union of all of the $(n-1)$-planes in $\PP[2n]$ that intersect $\E$ in exactly $n$ points.  It follows from the results in \cite[Section 8]{Gualtieri2013a} that the generic singular points are elliptic of type $\tE_6$.

\label{sec:ell-fano}
The rest of the paper is devoted to the proof of the following theorem (stated in the introduction) which characterizes these Poisson structures in the four-dimensional case:
\fanos*

The proof is organized as follows: in \autoref{sec:fano-proof-num}, we apply the topological constraint of \autoref{prop:678charclass} to the manifolds $\X$ in question.  In so doing, we immediately rule out the existence of purely elliptic log symplectic structures on the smooth quadric or on a product $\PP[1] \times \W$.  The remaining cases---$\PP[4]$ and cubic fourfolds---are dealt with in \autoref{sec:ell-proj} and \autoref{sec:ell-cubic}, respectively.

For the rest of the paper, $(\X,\D,\omega)$ is a purely elliptic log symplectic fourfold, and $\Y_6,\Y_7,$ and $\Y_8$ are the loci of $\tE_6$, $\tE_7$ and $\tE_8$ singularities, respectively.

\subsection{Numerical constraints}
\label{sec:fano-proof-num}

Suppose given an integral cohomology class $H \in \cohlgy[2]{\X,\ZZ}$.  We define integers
\[
a_i = \deg_H \Y_i = \int_{\Y_i} H
\]
for $i = 6,7,8$.  Thus each $a_i$ is the total degree of a collection of elliptic curves in $\X$, and from \autoref{prop:678charclass}, we obtain the equation
\begin{align}
8a_6 + 9a_7 + 10a_8 = H(c_1c_2-c_3) \cap [\X] \in \ZZ \label{eqn:ell-num}
\end{align}
If $H$ is a nef class, then each of the integers $a_6,a_7$ and $a_8$ is nonnegative.  We now specialize to the manifolds in \autoref{thm:p4-ell}.

\subsubsection*{Hypersurface case}

Let $\X \subset \PP[5]$ be a smooth hypersurface of degree $d \le 3$, and let $H \in \cohlgy[2]{\X,\ZZ}$ be the hyperplane class.   
Using the exact sequence for for the normal bundle, we have the total Chern class $c(\X) = (1+H)^6(1+dH)^{-1}$, from which one readily computes that
\[
c_1c_2-c_3 = (6d^2-36d+70)H^3.
\]
Meanwhile, $H^4 \cap [\X] = d$, and hence \eqref{eqn:ell-num} reads
\begin{align}
8a_6+9a_7+10a_8 = 6d^3-36d^2+70d \label{eqn:hypsurf-num}
\end{align}
The nonnegative integer solutions of this equation are displayed in \autoref{tab:hypsurf-sols}.  Since an elliptic curve in projective space has degree at least three, the solutions for $(a_6,a_7,a_8)$ with $d=2$ cannot be degrees of collections of elliptic curves.  Therefore a smooth quadric fourfold admits no purely elliptic structures.

\begin{table}[h]
\begin{center}
\caption{Nonnegative integer solutions of \eqref{eqn:hypsurf-num} with $d \le 3$}\label{tab:hypsurf-sols}
\begin{tabular}{c|c}
$d$ & $(a_6,a_7,a_8)$  \\
\hline 
1 & $(5,0,0)$ or $(0,0,4)$  \\
\hline
2 & $(1,4,0)$, $(2,2,1)$ or $(3,0,2)$ \\ \hline
3 & $(6,0,0)$
\end{tabular}
\end{center}
\end{table}

\subsubsection*{Product case}

Now suppose that $\X = \PP[1]\times \W$, where $\W \subset \PP[4]$ is a smooth hypersurface of degree $d < 5$, i.e.~a hypersurface that is Fano.  Let $A,B \in \cohlgy[2]{\X,\ZZ}$ be the pullbacks of the hyperplane classes on $\PP[1]$ and $\PP[4]$, respectively.  Thus $A^2 = 0$, $B^4 =0$ and $AB^3  \cap [\X]= d$.  The total Chern class of $\X$ is given by the formula $c(\X) = (1+A)^2(1+B)^5(1+dB)^{-1}$.  We conclude that
\[
c_1c_2-c_3 = 2(d-5)^2AB^2 + (5d^2-25d+40)B^3.
\]
Considering \eqref{eqn:ell-num} with the nef classes $A$ and $B$ we obtain the equations
\begin{align}
\begin{aligned}
8a_6+9a_7+10a_8  &= 2d(d-5)^2 \\
8b_6+9b_7+10b_8 &= d(5d^2-25d+40)
\end{aligned}\label{eqn:ab-eqns}
\end{align}
for the bidegrees $(a_i,b_i)$ of the hypothetical collections $\Y_6,\Y_7$ and $\Y_8$ of elliptic curves.
The solutions of these equations for which $a_i$ and $b_i$ are nonnegative integers are shown in \autoref{tab:numsols}.
\begin{table}[h]
\begin{center}
\caption{Nonnegative integer solutions of \eqref{eqn:ab-eqns}}\label{tab:numsols}
\begin{tabular}{c|c|c}
$d$ & $(a_6,a_7,a_8)$ & $(b_6,b_7,b_8)$ \\
\hline 
1 & $(0,0,2)$ & $(4,0,0)$  \\
\hline
2 & $(0,4,0)$, $(1,2,1)$ or $(2,0,2)$ &  $(0,0,2)$ \\
 \hline
3 & $(3,0,0)$ & $(0,0,3)$ \\
\hline
4 & $(1,0,0)$ & many
\end{tabular}
\end{center}
\end{table}

We claim that none of these solutions can represent the bidegrees of collections of elliptic curves, and hence  there can be no purely elliptic log symplectic structures on $\X$.

 Indeed, if $d=1,2$ or $3$, then at least one of the loci $\Y_i$ has degree zero with respect to $B$, but positive degree with respect to $A$. Such a curve is necessarily a (nonempty) union of fibres of the projection $\X \to \W$, which are copies of $\PP[1]$.  Hence they are not elliptic curves.  
 
 Likewise, if $d=4$, then $\Y_6$ has degree one with respect to $A$.  Since $A$ is nonnegative on every connected component of $\Y_6$, there must be a connected component $\C\subset \Y_6$ that has degree one with respect to $A$.  But then $\C$ would map isomorphically onto $\PP[1]$ under the projection $\X \to \PP[1]$, and hence it is not elliptic.  

\subsection{Projective space}
\label{sec:ell-proj}

We now consider purely elliptic log symplectic structures on $\PP[4]$.  Let $\D$ be the degeneracy hypersurface, and $\Y$ the locus where the Poisson structure vanishes.  Consulting \autoref{tab:hypsurf-sols}, we see that $\Y$ has degree at most five. Since each connected component is an elliptic curve, it must have degree at least three.  Hence $\Y$ must be connected, and there are two possibilities: either $\Y$ has degree five, in which case $\D$ has $\tE_6$ singularities; or $\Y$ has degree four, in which case $\D$ has $\tE_8$ singularities.    We now prove that  the case of a curve of degree four is impossible, while in the degree five case the only possibilities are the Feigin--Odesskii Poisson structures of type $q_{5,1}$

\subsubsection{Curve of degree five}

We first assume that the degree of $\Y$ is equal to five, so that $\D$ has $\tE_6$ singularities.  In this case we have the
\begin{lemma}
The hypersurface $\D \subset \PP[4]$ must contain the secant variety $\Sec{\Y}$.
\end{lemma}

\begin{proof}
In light of the local normal form for an $\tE_6$ singularity (see \autoref{tab:e678}) the section $s=\ps^2 \in \cohlgy[0]{\PP[4],\acan[{\PP[4]}]}$ that cuts out $\D$ must vanish to order three at every point of $\Y$.  Suppose that $p,q \in \Y$ are distinct points, and $\bL\subset \PP[4]$ is the secant line that joins them.  Then the restriction of $s$ to $\bL$ is a section of $\acan[{\PP[4]}]|_\bL(-3p-3q)\cong \sO{\bL}(-1)$, and hence it is identically zero, so that $\bL\subset \D$.  It follows that every secant line of $\Y$ is contained in $\D$, as required.
\end{proof}

\begin{corollary}
$\Y$ is an elliptic normal curve, and $\D = \Sec{\Y}$
\end{corollary}

\begin{proof}
If $\Y$ is not contained in a hyperplane, it is an elliptic normal curve and its secant variety $\Sec{\Y}$ is an irreducible quintic hypersurface.  Since $\D$ is a quintic containing $\Sec{\Y}$ we must have $\D = \Sec{\Y}$.

On the other hand, we claim that if $\Y$ is contained in a hyperplane  $\H \subset \PP[4]$, then $\Sec{\Y} = \H$, which, by the previous lemma, contradicts the fact that $\D$ is irreducible.  To see this, we note that, according to \cite{Ellingsrud1981,Hulek1983}, the curve $\Y$ is the linear projection to $\H$ of an elliptic normal curve $\Y' \subset \PP[4]$ from a point $p \in \PP[4] \setminus \Sec{\Y'}$.  Under this projection, every secant line of $\Y'$ maps to a secant line of $\Y$, and hence $\Sec{\Y}$ is the projection of $\Sec{\Y'}$, which is the whole hyperplane $\H$.
\end{proof}

We have now seen that the degeneracy divisor must be the secant variety of an elliptic normal curve.  To complete the classification in the degree five case, it remains to prove the  following
\begin{proposition}
Let $\Y \subset \PP[4]$ be an elliptic normal curve and $\D = \Sec{\Y}$ its secant variety.  Then the only purely elliptic log symplectic structures on $(\PP[4],\D)$ are the ones in the Feigin--Odesskii family $q_{5,1}$.
\end{proposition}

\begin{proof}
Since $\cohlgy[0]{\PP[4],\forms[2]{\PP[4]}}=0$,  logarithmic two-forms on $\PP[4]$ are uniquely determined by their residues.  So, it is enough to show that there is, up to rescaling, only one possibility for the residue of a purely elliptic log symplectic form on $(\PP[4],\D)$.  To do so, we use the well-known resolution of singularities of the secant variety, which we learned from \cite{GrafvBothmer2004}.

Let $\Y^{[2]}$ be the second symmetric power of $\Y$.  Then there is a natural rational map $\D \to \Y^{[2]}$ that sends a point $x \in \D$ contained in a secant line $\bL$ to the degree-two divisor defined by the intersection $\bL\cap \Y $.  This map extends to a regular map from the blowup $\tD$ of $\D$ along $\Y$, giving a $\PP[1]$-bundle  $\tD \to \Y^{[2]}$.

From the elliptic normal form \eqref{eqn:elliptic-C4} and the discussion in \autoref{ex:e6tilde}, it is clear that the residue of the log symplectic form must extend to a holomorphic one-form on $\tD$.  But the map that takes a degree-two divisor to its linear equivalence class makes $\Y^{[2]}$ into a $\PP[1]$-bundle over $\Pic[2]{\Y}\cong \Y$, and hence the composite map $\pi : \tD \to \Y^{[2]} \to \Y$ is an iterated $\PP[1]$ bundle.  It follows that the only holomorphic forms on $\tD$ are sections of $\pi^*\forms[1]{\Y} \cong \sO{\tD}$.  Since $\tD$ is connected and projective, the space of such sections is one-dimensional, as claimed.
\end{proof}

\subsubsection{Curve of degree four}

The possibility of any other purely elliptic log symplectic structure on $\PP[4]$ is ruled out by the following
\begin{theorem}
There does not exist a purely elliptic log symplectic structure on $\PP[4]$ whose singular locus $\Y$ is a curve of degree four.
\end{theorem}

\begin{proof}
Suppose to the contrary that such a structure exists, and let $\D$ be its degeneracy divisor.  Since $\D$ is normal, it must be irreducible.

Notice that the elliptic curve $\Y$ is necessarily the complete intersection of a hyperplane $\H \subset \PP[4]$ and two quadrics $\Q_1,\Q_2 \subset \PP[4]$.  Let $h$ be a defining equation for the hyperplane, and extend $h$ to a system $(h,x_0,x_1,x_2,x_3)$ of homogeneous coordinates for $\PP[4]$.  Without loss of generality, we may assume that the quadratic forms $q_1,q_2$ defining $\Q_1$ and $\Q_2$ depend only on the coordinates $x_0,\ldots,x_3$.  Then, as is well known, every quadric in the pencil spanned by $q_1$ and $q_2$ has rank equal to three or four.

A theorem of Bondal~\cite{Bondal1993} asserts that the Poisson bracket on $\PP[4]$ has a canonical lift to a Poisson bracket on the homogeneous coordinate ring $\CC[h,x_0,\ldots,x_3]$, such that the elementary brackets $\{h,x_i\}$ and $\{x_i,x_j\}$ are given by homogeneous quadratic forms.  
To derive a contradiction, we will show by direct calculation that all Poisson brackets of the form $\{h,x_i\}$ must be multiples of $h$, i.e.~that the ideal $(h)$ is a Poisson ideal.  This implies that the hyperplane $\H \subset \PP[4]$ is a Poisson subspace.  Since $\dim \H = 3$, the Poisson tensor must have rank at most three on $\H$, so that $\H \subset \D$, which contradicts the irreducibility of $\D$.

In our calculation, we will require two basic properties of the bracket:
\begin{enumerate}
\item The ideal $\sI = (h,q_1,q_2)$ cutting out $\Y$ is a Poisson ideal, i.e. $\{\sI,\sO{\CC^5}\} \subset \sI$.
\item The homogeneous quintic polynomial $f$ defining $\D$ is a Casimir function, i.e.~$\{f,\sO{\CC^5}\} = 0$.
\end{enumerate}
These properties can easily be established using the definition of the quadratic Poisson bracket via the canonical Poisson module structure on $\acan[{\PP[4]}] \cong \sO{\PP[4]}(5)$ in \cite[Section 12]{Polishchuk1997}; the first property follows from the fact that $\Y \subset \PP[4]$ is a strong Poisson subspace in the sense of \cite{Gualtieri2013a}, and the second follows from the fact that the anticanonical section cutting out $\D \subset \PP[4]$ is a Poisson flat section~\cite[Proposition 7.4]{Polishchuk1997}.

 Since $\sI$ is a Poisson ideal, the elementary brackets must have the form
\begin{align}
\{h,x_i\} = s_i q_1 + t_i q_2 \mod h \label{eqn:st-def}
\end{align}
for constants $s_i,t_i \in \CC$.  We are therefore required to show that $s_i = t_i = 0$ for all $i$.  It will be useful to define linear forms $A,B,C,D$ by the formulae
\begin{align}
\begin{split}
\{h,q_1\} &= Aq_1 + Bq_2 \mod h \\
\{h,q_2\} &= Cq_1 + Dq_2 \mod h.
\end{split}\label{eqn:ABCDdef}
\end{align}

Now we observe that, since $\D$ is singular along the complete intersection $\Y$, the defining equation for $\D$ may be written as
\[
f = \tfrac{1}{2}a_{11} q_1^2 + a_{12} q_1q_2 + \tfrac{1}{2} a_{22}q_2^2  + (b_1q_1+b_2q_2)h \mod h^2,
\]
for quadratic forms $b_1,b_2$ and linear forms $a_{11},a_{12},a_{22}$ in the variables $x_0,\ldots,x_3$.  Since $\D$ has $\tE_8$ singularities along $\Y$, the Hessian of $f$ must have rank one on the cone over $\Y$ in $\CC^5$.  Hence the two-by-two minors of the Hessian must lie in $\sI$.  In particular, the expressions
\begin{align*}
M &= a_{11}a_{22}-a_{12}^2 & N &= b_1a_{12}-b_2a_{11}
\end{align*}
must lie in $\sI$.  Therefore $M$ is a quadratic form in the pencil spanned by $q_1$ and $q_2$.  Evidently, $M$ has rank at most three, and so there are exactly two possibilities: either $M$ is identically zero, or it has rank equal to three.  We now consider these two cases separately.

\paragraph*{Case 1:} $M=0$.  In this case, the linear forms $a_{11}$, $a_{12}$ and $a_{22}$ must all be constant multiples of a single form $a$.  Using the equation $M=0$, we may factor the first three terms of $f$ and write
\[
f = \tfrac{1}{2}a q_1^2 + (b_1q_1+b_2q_2)h  \mod h^2
\]
without loss of generality.  Since $\D$ is irreducible, $f$ is not divisible by $h$, and hence $a \ne 0$.

The minor $N \in \sI$ defined above is now given by $N = ab_2$.  Since $a$ is not a zero divisor modulo $\sI$, this forces $b_2$ to be a linear combination of $q_1$ and $q_2$.  If $b_2$ were a multiple of $q_1$, we would have $f \in (h,q_1)^2$, which would imply that $\D$ is singular along the quadric surface defined by $h$ and $q_1$, contradicting the normality of $\D$.  Therefore $b_2$ is linearly independent from $q_1$, and so we may assume without loss of generality that
\begin{align}
f = \tfrac{1}{2}a q_1^2 + bq_1h+\tfrac{1}{2}q_2^2h \mod h^2 \label{eqn:f-simpl}
\end{align}
for some quadratic form $b$.

We now compute the Poisson bracket $\{h,f\} = 0$ modulo $h$, giving the equation
\[
\tfrac{1}{2}\{h,a\}q_1^2 = - a(Aq_1+Bq_2)q_1 \mod h.
\]
Since $\sI$ is a Poisson ideal, the left-hand side of this equation lies in $\sI^3$.  Since $\sI$ is a complete intersection and $a$ is not a zero divisor modulo $\sI$, this forces $A=B=0$.  Therefore the brackets $\{h,q_1\}$ and $\{h,a\}$ are multiples of $h$.  In particular, we may write
\begin{align*}
\{h,q_1\} &= hQ \mod h^2
\end{align*}
for some $Q$ depending only on $x_0,\ldots,x_3$.

If the quadratic form $q_1$ has rank four, we may choose the coordinates so that $q_1 = \tfrac{1}{2}(x_0^2+x_1^2+x_2^2+x_3^3)$.  Using \eqref{eqn:st-def}, we compute the bracket
\[
\{h,q_1\} = \rbrac{\sum_{i=0}^3s_ix_i}q_1 + \rbrac{\sum_{i=0}^3t_ix_i}q_2 \mod h
\]
Since $\sI$ is a complete intersection and $\{h,q_1\}$ is divisible by $h$, we must have $s_i=t_i = 0$ for all $i$ and hence the ideal $(h)$ is Poisson.

If, on the other hand, the form $q_1$ has rank three, we may choose the  coordinates so that $q_1 = \tfrac{1}{2}(x_0^2+x_1^2+x_2^2)$ and $q_2 = \tfrac{1}{2}x_2^2 \mod (x_0,x_1,x_2)^2$.  Calculating as above, we find that $s_i=t_i = 0$ for $0 \le i \le 2$.  Hence the Hamiltonian vector field of $h$ is given by
\[
\{h,\cdot\} = (sq_1+tq_2)\cvf{x_3} \mod h \cdot \der[1]{\CC^5}
\]
where $s=s_3$ and $t=t_3$.  It remains to show that $s$ and $t$ are equal to zero.  Using \eqref{eqn:f-simpl}, we compute
\[
0 =h^{-1}\{h,f\} = aQq_1 + (t\cvf{x_3}b+sx_3)q_1q_2   + (tx_3)q_2^2  \mod (h,q_1^2).
\]
Since  $x_3$ is not a zero divisor modulo $(h,q_1)$, we must have $t=0$.  Moreover, since all terms but the first lie in $(q_1,q_2)^2$, we must have that $Q$ is a linear combination of $q_1$ and $q_2$, say $Q = \lambda_1q_1 + \lambda_2q_2$.  Therefore the equation above reduces to
\[
(sx_3+\lambda_2a)q_1q_2 = 0 \mod (h,q_1^2) 
\]
which implies that $sx_3 = -\lambda_2 a$.  If $s$ were nonzero, we would have
\[
\{h,a\} = \{h,-s\lambda_2^{-1}x_3\} = -s^2\lambda_2^{-1} q_1 \mod h
\]
which contradicts the fact that $\{h,a\} = 0 \mod h$.  Hence $s$ must be zero, completing the proof that $(h)$ is a Poisson ideal when $M=0$.

\paragraph*{Case 2:} $M$ has rank three.  In this case, the linear forms $x=a_{11}$, $y=a_{22}$ and $z=a_{12}$ must be linearly independent, and hence we may use them as the homogeneous coordinates $x_0,x_1$ and $x_2$

The quadric $M = xy-z^2$ lies in $\sI$ and hence it is a linear combination of $q_1$ and $q_2$.  After a change of basis in the pencil spanned by $q_1$ and $q_2$, we may assume that $M = q_1$.  We then choose our final homogeneous coordinate $w=x_3$ so that
\[
q_2 = w^2  + Q(x,y,z)
\]
for a quadratic form $Q$ in three variables.

Computing the bracket $\{h,f\} = 0$, we find the equation
\begin{align}
\begin{split}
\tfrac{1}{2}\{h,x\}q_1^2 + \{h,z\}q_1q_2+\tfrac{1}{2}\{h,y\}q_2^2 
&= - (xB+z(A+D)+yC)q_1q_2 \\
&\quad -(Ax+zC)q_1^2 \\
&\quad - (yD+zB)q_2^2 \mod h
\end{split}\label{eqn:ABCD}
\end{align}
Since $\sI$ is a Poisson ideal, the left hand side of this equation lies in $\sI^3$.  Thus the coefficients of $q_1^2,q_1q_2$ and $q_2^2$ on the right-hand side must be linear combinations of $q_1$ and $q_2$ modulo $h$.  But these coefficients are quadratic forms that lie in the ideal $(x,y,z)$ defining the critical locus of $q_1=xy-z^2$.  Since the critical locus is not in the base locus of the pencil, these forms must all be multiples of $q_1$.  This puts strong constraints on $A,B,C$ and $D$; one can easily compute that they must be written
\begin{align*}
A &=  Fy+Gz &
B&= - Uy-Vz &
C &=  -Gx-Fz &
D &= Vx+Uz
\end{align*}
for some constants $F,G,U,V \in \CC$.
Then, using the form~\eqref{eqn:st-def} of the elementary brackets, the definition \eqref{eqn:ABCDdef} of $A$ and $B$,  and the formula $q_1=xy-z^2$ one finds 
\begin{align}
\begin{split}
\{h,x\} &= Fq_1 - Uq_2 \mod h \\ 
\{h,y\} &= 0  \mod h\\ 
\{h,z\} &= -\tfrac{1}{2}Gq_1+\tfrac{1}{2}Vq_2 \mod h. 
\end{split}\label{eqn:hxyz}
\end{align}
Equation \eqref{eqn:ABCD} now gives
\[
Fq_1^3+(U+G)q_1^2q_2+Vq_1q_2^2 = 0 \mod h
\]
which implies that $F=V=0$ and $U=-G$.  Hence $C=-Gx$ and $D=-Gz$.  Using the definition \eqref{eqn:ABCDdef} of $C$ and $D$, and the brackets~\eqref{eqn:hxyz}, we find
\begin{align*}
Gq_2\cvf{x}Q -\tfrac{1}{2}Gq_1\cvf{z}Q + 2\{h,w\}w &= -Gxq_1-Gzq_2 \mod h
\end{align*}
which is easily seen to imply that $G=0$, and that $\{h,w\} = 0$ modulo $h$.  It follows that $h$ generates a Poisson ideal, as required.
\end{proof}

\subsection{Cubic fourfolds}
\label{sec:ell-cubic}
In order to complete the proof of \autoref{thm:p4-ell}, it remains to show that a smooth cubic fourfold $\X$ does not admit any purely elliptic log symplectic forms.  From \autoref{tab:hypsurf-sols}, we know that the degeneracy divisor $\D \subset \X$ of such a structure would have $\tE_6$ singularities along a locus $\Y \subset \X$ that is a union of elliptic curves of total degree six.  Hence $\Y$ must either be a single curve of degree six or a pair of plane cubics.   In \autoref{prop:cubic-6} and \autoref{prop:cubic-33} below, we will show that such a hypersurface cannot exist.  We begin by reducing the problem to the study of singular cubic fourfolds:
\begin{lemma}\label{lem:mult3}
Let $\X \subset \PP[5]$ be a smooth cubic fourfold and $\D \subset \X$ an anticanonical divisor.  Suppose that $\C \subset \D_{sing}$ is a submanifold along which $\D$ has multiplicity three.  Then there is a unique cubic fourfold $\tD \subset \PP[5]$ such that $\D = \tD \cap \X$ and $\tD$ has multiplicity three along $\C$.  Moreover $\tD$ contains the variety $\Sec[2]{\C}$ swept out by the secant planes of $\C$.
\end{lemma}

\begin{proof}
Let $\sJ \subset \sO{\X}$ and $\sI \subset \sO{\PP[5]}$ be the ideals defining $\C$ as a subspace of $\X$ and $\PP[5]$, respectively, and let $\sO{\PP[5]}(-\X)\subset \sO{\PP[5]}$ be the ideal defining $\X$.  By the adjunction formula, we have $\acan[\X] \cong \sO{\X}(3)$, and so the anticanonical divisors with multiplicity three along $\C$ are cut out by sections of $\sJ^3(3)$.  To prove the existence and uniqueness of $\tD$, we must show that the map
\[
\xymatrix{
\cohlgy[0]{\sI^3(3)} \ar[r] & \cohlgy[0]{\sJ^3(3)}
}
\]
induced by the restriction $\sO{\PP[5]} \to \sO{\X}$ is an isomorphism.

Since $\X$ and $\C$ are smooth, the restriction gives a surjection $\sI^3 \to \sJ^3$ with kernel $\sI^3 \cap \sO{\PP[5]}(-\X)$, and this kernel is the precisely the image of the multiplication map $\sI^2 \otimes \sO{\PP[5]}(-\X) \to \sI^3$.  Twisting by $\sO{\PP[5]}(3)$, we obtain an exact sequence
\[
\xymatrix{
0 \ar[r] & \sI^2 \ar[r] & \sI^3(3) \ar[r] & \sJ^3(3)\ar[r] & 0
}
\]
and the relevant part of the long exact sequence reads
\[
\xymatrix{
\cohlgy[0]{\sI^2} \ar[r] & \cohlgy[0]{\sI^3(3)} \ar[r] & \cohlgy[0]{\sJ^3(3)} \ar[r] & \cohlgy[1]{\sI^2}.
}
\]
Now $h^0(\sI^2) = 0$ since $\PP[5]$ has no nonconstant global functions.  Hence the lemma will follow if we can show that $h^1(\sI^2)=0$ as well.  For this, we consider exact sequence
\[
\xymatrix{
0 \ar[r] & \sI^2 \ar[r] & \sI \ar[r] & \sN^\vee \ar[r] & 0
}
\]
defining the conormal sheaf $\sN^\vee$ of $\C$ in $\PP[5]$.  Since $\sN^\vee \subset \forms[1]{\PP[5]}|_\C \subset \sO{\C}(-1)^{\oplus 6}$, we have $h^0(\sN^\vee)=0$.  Hence the vanishing of $h^1(\sI^2)$ follows from the vanishing of $h^1(\sI)$, which in turn follows from the exact sequence
\[
\xymatrix{
0 \ar[r] & \sI \ar[r] & \sO{\PP[5]}\ar[r] & \sO{\X}\ar[r] & 0
}
\]
and the vanishing of $h^1(\sO{\PP[5]})$.  This completes the proof of the existence and uniqueness of $\tD$.

To see that $\Sec[2]{\C} \subset \tD$, let $\W \subset \PP[5]$ be a plane that hits $\C$ at three points $p,q,r \in \C$ that are not collinear.  Then either $\W \subset \Y$ or the intersection $\W \cap \Y$ is a cubic curve with multiplicity three at $p,q$ and $r$.  But the latter is impossible because the only cubic curve with three non-collinear singular points is a triangle, for which the singularities are nodes.
\end{proof}

\begin{proposition}\label{prop:cubic-6}
Let $\X \subset \PP[5]$ be a smooth cubic fourfold.  If $\D \subset \X$ is an anticanonical divisor having multiplicity three along an elliptic curve of degree six, then $\D$ is not normal.
\end{proposition}

\begin{proof}
Let $\Y \subset \X\subset \PP[5]$ be a degree six elliptic curve. In light of \autoref{lem:mult3} it is enough to show that a cubic fourfold $\tD$ with multiplicity three along $\Y$ cannot be normal.  

To this end, let $\W \subset \PP[5]$ be the smallest linear subspace containing $\Y$.  We have $h^0(\sO{\Y}(1)) = 6$ by Riemann--Roch, and hence $\Y$ is the linear projection of an elliptic normal sextic $\Y' \subset \PP[5]$ to $\W$.  Since the secant planes of an elliptic normal sextic fill out all of $\PP[5]$, we must have that $\W = \Sec[2]{\Y}$.  Hence $\W \subset \tD$ by \autoref{lem:mult3}. It follows that the dimension of $\W$ is at most four.   
On the other hand, the dimension of $\W$ is at least three because $\PP[2]$ does not contain an elliptic sextic. 

If $\W$ has dimension four, then $\W$ is an irreducible component of $\tD$.  Hence $\tD$ is not normal.

So, suppose that $\W$ has dimension three.  We may choose homogeneous coordinates $x_0,x_1,y_0,\ldots,y_3$ for $\PP[5]$ such that $\W$ is cut out by the equations $x_0=x_1=0$.  Since $\tD$ contains $\W$, it is cut out by a cubic polynomial of the form
\[
f = x_0q_0 + x_1 q_1  \mod (x_0,x_1)^2
\]
where $q_0,q_1$ are quadratic  forms in the variables $y_0,\ldots,y_3$.  Now $f$ vanishes to order three on $\Y$.  Hence the derivative
\[
df = q_0\, dx_0 + q_1 \, dx_1 \mod (x_0,x_1)
\]
and the Hessian
\[
\hess{f} = dq_0 \cdot dx_0 + dq_1 \cdot dx_1  \mod (x_0,x_1,(dx_0)^2,dx_0\cdot dx_1,(dx_1)^2) 
\]
must vanish on $\Y$.  This implies that the quadrics defined by $q_0$ and $q_1$ are singular along $\Y$.  But the singular locus of a quadric is a linear subspace, and $\W$ is the smallest linear subspace containing $\Y$.  Hence $q_0=q_1=0$ identically, which implies that $f \in (x_0,x_1)^2$.  Hence $\tD$ is singular along $\W$, which implies that $\tD$ is not normal.
\end{proof}

\begin{proposition}\label{prop:cubic-33}
Let $\X \subset \PP[5]$ be a smooth cubic fourfold.  If $\D \subset \X$ is an anticanonical divisor having multiplicity three along a pair $\Y_1,\Y_2 \subset \X$ of disjoint cubic elliptic curves, then $\D$ is not normal.
\end{proposition}

\begin{proof}
Because of the degrees, the curves $\Y_1,\Y_2$ are contained in planes $\W_1,\W_2$.  Let $\W_i^\perp \subset \cohlgy[0]{\sO{\PP[5]}(1)}$ be the space of linear forms cutting out $\W_i$.  Because $\Y_i = \X \cap \W_i$ is a complete intersection, the cubic fourfolds with multiplicity three along $\Y_i$ are given by the subspace
\[
\sym[3]{\W_i^\perp} \subset \cohlgy[0]{\sO{\PP[5]}(3)}.
\]
Hence if $\tD \subset \PP[5]$ is the singular cubic fourfold provided by \autoref{lem:mult3}, its defining equation lies in the intersection
\[
\sym[3]{\W_1^\perp} \cap \sym[3]{\W_2^\perp} = \sym[3](\W_1^{\perp} \cap \W_2^\perp)
\]
Therefore $\tD$ has multiplicity three along the linear span of $\W_1$ and $\W_2$ in $\PP[5]$, which implies that the singular locus of $\tD$ has codimension at most $\dim(\W_1\cap \W_2)$ in $\tD$.  Thus $\tD$ can only be normal if $\W_1=\W_2$, but then the cubic curves $\Y_1$ and $\Y_2$ lie in the same plane, and hence they cannot be disjoint.
\end{proof}

\bibliographystyle{hyperamsplain}
\bibliography{ell-log-symp}
\end{document}